\documentclass[12pt]{amsart}

\usepackage{amsfonts}
\usepackage{amsmath}
\usepackage{amssymb}
\usepackage{amsthm}
\usepackage{enumerate} 
\usepackage{hyperref}
\hypersetup{
	colorlinks=true,
	citecolor=blue
}
\usepackage[margin=1.2in]{geometry}
\usepackage{mathrsfs} 
\usepackage{graphicx} 
\usepackage{tikz-cd}

\usepackage{xcolor}

\theoremstyle{plain}
\newtheorem{theorem}{Theorem}[section]
\newtheorem{lemma}[theorem]{Lemma}
\newtheorem{proposition}[theorem]{Proposition}
\newtheorem{corollary}[theorem]{Corollary}

\theoremstyle{definition}

\newtheorem{remark}[theorem]{Remark}

\newtheorem{example}[theorem]{Example}

\numberwithin{equation}{section}

\newcommand\N{\mathbb{N}}
\newcommand\Z{\mathbb{Z}}
\newcommand\R{\mathbb{R}}

\newcommand\C{\mathbb{C}}

\newcommand\SES[5]{\begin{tikzcd}[ampersand replacement=\&] 0 \arrow{r} \& #1 \arrow{r}{#4} \& #2 \arrow{r}{#5} \& #3 \arrow{r} \& 0\end{tikzcd}}

\DeclareMathOperator\im{im}
\DeclareMathOperator\id{id}

\DeclareMathOperator\DN{DN}

\DeclareMathOperator\Proj{Proj}

\DeclareMathOperator\Ext{Ext}

\DeclareMathOperator\LS{LS}
\newcommand\ExtLS{\Ext_{\LS}}
\DeclareMathOperator\PLS{PLS}
\DeclareMathOperator\PLB{PLB}
\DeclareMathOperator\LB{LB}
\newcommand\ExtPLS{\Ext_{\PLS}}

\newcommand\compl{\subseteq_{c}}

\begin{document}

\title[A Pe\l czy\'nski-Vogt decomposition result for $(\PLS)$-spaces]{
A Pe\l czy\'nski-Vogt decomposition result for $(\PLS)$-spaces and sequence space representations}

\author[A. Debrouwere]{Andreas Debrouwere}
\address{Department of Mathematics and Data Science \\ Vrije Universiteit Brussel, Belgium\\ Pleinlaan 2 \\ 1050 Brussels \\ Belgium}
\email{andreas.debrouwere@vub.be}

\author[L. Neyt]{Lenny Neyt}
\address{University of Vienna\\ Faculty of Mathematics\\ Oskar-Morgenstern-Platz 1 \\ 1090 Wien\\ Austria}
\thanks{The research of L. Neyt was funded in whole by the Austrian Science Fund (FWF) 10.55776/ESP8128624. For open access purposes, the author has applied a CC BY public copyright license to any author-accepted manuscript version arising from this submission.}
\email{lenny.neyt@univie.ac.at}


\begin{abstract}
We establish a Pe\l czy\'nski-Vogt decomposition result for $(\PLS)$-type power series spaces of infinite type. By combining this result with the theory of 
Gabor frames, we obtain sequence space representations for multiplier spaces of Gelfand-Shilov spaces of Beurling type. 
\end{abstract}

\maketitle  

\section{Introduction}
\subsection{A Pe\l czy\'nski-Vogt decomposition result for (PLS)-spaces} Given two locally convex spaces $X$ and $Y$, we write $X \cong Y$  if $X$ and $Y$ are isomorphic, and $X \compl Y$ if $X$ is isomorphic to a complemented subspace of $Y$. 

In his study of the complemented subspaces of the classical Banach sequence spaces $\ell_p$  and $c_0$, Pe\l czy\'nski showed the following decomposition result \cite[Proposition 4]{Pelczynski} (see also \cite[Theorem 2.2.4]{Albiac-Kalton-TopicsBanach}): \emph{Let $Y$ = $\ell_p$ $(1 \leq p < \infty)$ or $c_0$}.
\begin{equation}
\label{PVstat}
\mbox{\emph{ Let $X$ be a locally convex space with $X \compl Y$ and $Y \compl X$. Then, $X \cong Y$.}}
\end{equation}
The proof of this result is elementary and relies on the identities $\ell_p(\ell_p) \cong \ell_p$ and $c_0(c_0) \cong c_0$.

In the context of sequence space representations,  Vogt later obtained the following variant of Pe\l czy\'nski's result \cite[Proposition 1.2]{V-SeqSpRepTestFuncDist}:  \emph{Let $E$ be a locally convex space and set $Y = E^{\N}, E^{(\N)}$, or $s(E)$. Then, \eqref{PVstat} holds.} The proof of this result is very similar to the one of Pe\l czy\'nski, now using that $Y \varepsilon Y \cong Y$. Applying the splitting theory for Fr\'echet spaces, \cite{firstsplittingVogt} (see also  \cite{V-Ext1Frechet}), Vogt \cite[Satz 1.4]{V-IsomorphSatzPotRaum} showed the much deeper result that \eqref{PVstat} holds if $Y = \Lambda_{\infty}(\alpha)$ is a stable nuclear power series space of infinite type. Moreover, \eqref{PVstat} is also true if $Y = \Lambda_{0}(\alpha)$ is a nuclear power series space of finite type \cite[Satz 1.4]{V-IsomorphSatzPotRaum}, as follows from a theorem of Mityagin, which states that every locally convex space $X$ with $X \compl \Lambda_0(\alpha)$ is isomorphic to a power series of finite type \cite[Theorem 1.1]{MH} (see also \cite[Corollary 29.20]{M-V-IntroFunctAnal}), together with basic properties of the diametral dimension. 
 
 Results of this form are often called \emph{Pe\l czy\'nski-Vogt decomposition results}. They were used by  Valdivia and Vogt in the 1980s to establish sequence space representations of most of the function and distribution spaces arising in Schwartz’s theory of distributions \cite{V-TopLCS, V-SeqSpRepTestFuncDist}; see \cite{B-D-N-SeqSpRepWilson,D-SeqSpRepEntFunc, D-N-SeqSpRepTMIB, L-BasesGerms, L-DiamDimWeighSpGerms}  for more recent, related works.
  
 $(\PLS)$-spaces are locally convex spaces that can be written as the countable projective limit of $(\LS)$-spaces. Important examples include the space of distributions, the space of real analytic functions, and the space $\mathcal{O}_M$ of slowly increasing smooth functions.
The class of $(\PLS)$-spaces plays an important role in the modern theory of locally convex spaces, particularly in connection with the derived projective limit functor \cite{W-DerivFuncFunctAnal}. We refer to the survey article \cite{Domanski-PLS} of Doma\'nski for more information on  $(\PLS)$-spaces.
 
 In the first part of this article, we prove a Pe\l czy\'nski-Vogt decomposition result for $(\PLS)$-type power series spaces of infinite type. Let  $\Lambda_{\infty}(\alpha)$ and  $\Lambda_{\infty}(\beta)$ be nuclear and assume that $\Lambda_{\infty}(\alpha)$ is stable. Consider the $(\PLS)$-space
 \begin{align*}
 	\Lambda_\infty(\alpha, \beta) 
 	&= \Lambda_{\infty}(\alpha) \varepsilon \Lambda'_{\infty}(\beta) \\
	&= \{ c = (c_{i,j})_{(i,j) \in \N^2} \mid \forall n \in \N \, \exists N \in \N :  \| c\|_{n,N} = \sup_{(i,j) \in \N^2} |c_{i,j}|e^{n\alpha_i -N \beta_j} < \infty \}.
 \end{align*}
 In Theorem \ref{t:PelczynskiPLS}, we show that \eqref{PVstat} holds for $Y =  \Lambda_\infty(\alpha, \beta)$. To this end, we extend the technique of Vogt's proof \cite[Satz 1.4]{V-IsomorphSatzPotRaum} of the decomposition result for power series of infinite type to the setting of $(\PLS)$-spaces. 
 
 One of the main ingredients in our proof is a new splitting result for  $(\PLS)$-type sequence spaces, namely, that every topologically exact sequence of $(\PLS)$-spaces
 $$
 \SES{ \Lambda_\infty(\alpha, \beta)}{X}{ \Lambda_\infty(\gamma, \delta)}{}{} 
$$
with $\Lambda_{\infty}(\gamma)$ and $\Lambda_{\infty}(\delta)$ nuclear, splits (see Theorem \ref{t:SplittingResult} for a more general result). In his PhD thesis \cite[Theorem 5.14]{K-SplitPowerSeriesSpPLS}, Kunkle showed a similar result, but for a different type of sequence spaces. We closely follow Kunkle’s approach to show the above result. For the reader’s convenience, and since the results from \cite{K-SplitPowerSeriesSpPLS} have never been published, we include here a full proof. See \cite{Domanski-splitting} and the references therein for other works related to the splitting theory for  $(\PLS)$-spaces.

We expect that an analogous decomposition should also hold in the finite-type case, namely, that \eqref{PVstat} is true for $ Y = \Lambda_{0}(\alpha) \varepsilon \Lambda'_{0}(\beta)$, with $\Lambda_{0}(\alpha)$ and $\Lambda_{0}(\beta)$ nuclear, however, we have not yet been able been able to prove this.

\subsection{Sequence space representations for multiplier spaces of Gelfand-Shilov spaces of Beurling type} In his classical book \cite[Chapitre VII, $\S $5]{Schwartzbook}, Schwartz introduced the space 
$$
 \mathcal{O}_M = \{ f \in C^\infty(\R) \mid \forall n \in \N \, \exists N \in \N :  \| f\|_{n,N} = \max_{p \leq n}  |f^{(p)}(x)|(1+|x|)^{-N}< \infty \}.
$$
and showed that it was equal to the multiplier space of the space $\mathcal{S}$ of rapidly decreasing smooth functions, that is, 
$$
 \mathcal{O}_M = \{ f \in \mathcal{S}' \mid \varphi \cdot f \in \mathcal{S} \mbox{ for all $\varphi \in \mathcal{S}$}\}.
$$
Moreover, the natural $(\PLS)$-space topology on  $\mathcal{O}_M$ coincides with the operator topology induced by the embedding
	\[ \mathcal{O}_{M} \rightarrow L_{b}(\mathcal{S}, \mathcal{S}), \quad f \mapsto ( \varphi \mapsto \varphi \cdot f) . \]	
	
In his doctoral thesis \cite[Chapitre II, Theor\`eme 16, p.\ 131]{G-ProdTensTopEspNucl}, Grothendieck proved that the space  $\mathcal{O}_M$ is ultrabornological. He achieved this by showing that $\mathcal{O}_M \compl s(s')$ and verifying directly that  $s(s')$ is ultrabornological. Later, Valdivia \cite{V-RepOM} showed that in fact $\mathcal{O}_{M} \cong s(s')$. To this end, he showed, in addition, that  $s(s') \compl \mathcal{O}_M$  and used the Pe\l czy\'nski-Vogt decomposition result for $s(s')$. We refer to \cite{B-D-N-SeqSpRepWilson} for a constructive proof of the isomorphism $\mathcal{O}_{M} \cong s(s')$.

In  \cite{D-N-WeighPLBUltradiffFuncMultSp, D-N-BarrelWeighPLBUltradiffFunc}, we studied the structural and linear topological properties of the multiplier space of Gelfand-Shilov spaces; see \cite{D-P-V-MultConvTempUltraDist, S-InclThMoyalMultAlgGenGSSp} for related works. In the second part of this article, we apply the Pe\l czy\'nski-Vogt decomposition result for the spaces $\Lambda_{\infty}(\alpha,\beta)$  to obtain sequence space representations for multiplier spaces of Gelfand-Shilov spaces of Beurling type. We now state a sample of this result.

Fix $\mu,\tau >0$. For $h >0$ and $\lambda \in \R$ we define the Banach space
$$
\Sigma^{\mu,h}_{\tau,\lambda} = \{ \varphi \in C^\infty(\R) \mid  \| f\| =  \sup_{p \in \N} \sup_{x \in \R} \frac{|\varphi^{(p)}(x)|e^{-\frac{1}{\lambda}|x|^{1/\tau}}}{h^{p}p!^\mu} < \infty \}.
 $$
Consider the Fr\'echet space
$$
\Sigma^{\mu}_{\tau}  = \varprojlim_{h \to 0^+} \Sigma^{\mu,h}_{\tau,h}  .
$$ 
The spaces $\Sigma^\tau_\mu$ are the projective analogues of the classical Gelfand-Shilov spaces $\mathcal{S}^\mu_\tau$ \cite{G-S-GenFunc2} and have been considered in e.g.\ \cite{C-G-P-R-AnistropicShubinOpExpGSSp, D-SeqSpRepEntFunc, D-N-SeqSpRepBBSp, Petersson}. We mention that $\Sigma^\tau_\mu$ is non-trivial if and only if $\mu + \tau >1$ (cf.\ \cite[Section 8]{G-S-GenFunc2}).

Define the $(\PLS)$-space
$$
\mathcal{Z}^{(\mu)}_{(\tau)} = \varprojlim_{h \to 0^+}  \varinjlim_{\lambda \to 0^-}  \Sigma^{\mu,h}_{\tau,\lambda}.   
$$
In \cite[Theorem 5.7(i)]{D-N-WeighPLBUltradiffFuncMultSp} we showed that $\mathcal{Z}^{(\mu)}_{(\tau)}$ is the multiplier space  of $\Sigma^{\mu}_{\tau}$, i.e., 
$$
\mathcal{Z}^{(\mu)}_{(\tau)} = \{ f \in (\Sigma^\tau_\mu)' \mid \varphi \cdot f \in\Sigma^{\mu}_{\tau}  \mbox{ for all $\varphi \in\Sigma^{\mu}_{\tau} $}\},
$$
and that the natural $(\PLS)$-space topology on  $\mathcal{Z}^{(\mu)}_{(\tau)}$ coincides with the topology induced by the embedding
	\[ \mathcal{Z}^{(\mu)}_{(\tau)} \rightarrow L_{b}(\Sigma^{\mu}_{\tau} ,\Sigma^{\mu}_{\tau} ), \quad f \mapsto ( \varphi \mapsto \varphi \cdot f) . \]	
Moreover, $\mathcal{Z}^{(\mu)}_{(\tau)}$ is ultrabornological  \cite[Theorem 5.7(ii)]{D-N-WeighPLBUltradiffFuncMultSp}.

Here, we obtain a  sequence space representation for $\mathcal{Z}^{(\mu)}_{(\tau)}$: 

\begin{theorem}\label{main-intro}
Let $\mu, \tau > 1/2$. Set $\alpha_\mu = (j^{1/\mu})_{j \in \N}$ and  $\alpha_\tau = (j^{1/\tau})_{j \in \N}$. Then,
$$
\mathcal{Z}^{(\mu)}_{(\tau)} \cong \Lambda(\alpha_\mu, \alpha_\tau).
$$
\end{theorem}
The proof of Theorem \ref{main-intro} combines the Pe\l czy\'nski-Vogt decomposition result for the $(\PLS)$-type power series spaces $\Lambda_{\infty}(\alpha,\beta)$ of infinite type with properties of Gabor frames---a fundamental tool in time-frequency analysis \cite{G-FoundTFAnalysis}. In \cite{D-N-SeqSpRepBBSp}, we used the same technique to obtain sequence space representations for Gelfand-Shilov spaces, in particular $\Sigma^{\mu}_{\tau}$. The restriction $\mu, \tau > 1/2$ is imposed because we rely on a result on dual windows of Gabor frames from \cite{Janssen}. 
Finally, we note that in  \cite{D-N-WeighPLBUltradiffFuncMultSp, D-N-BarrelWeighPLBUltradiffFunc} we also made use of time-frequency analysis, namely, the short-time Fourier transform \cite{G-FoundTFAnalysis}, which can be seen as a continuous version of Gabor frames.

\subsection{Outline of the paper} Section \ref{sect-prelim} collects notation and preliminary notions. In Section \ref{sect-seqsp}, we define and discuss various sequence spaces. We prove a splitting result for $(\PLS)$-type sequence spaces in Section \ref{sect-splitting}. This section is based on \cite{K-SplitPowerSeriesSpPLS}. Section \ref{sect-PV} is devoted to the proof of the  Pe\l czy\'nski-Vogt decomposition result for $(\PLS)$-type power series spaces. In Section \ref{sect-appl}, we apply this result to obtain sequence space representations for multiplier spaces of Gelfand-Shilov spaces of Beurling type. In particular, we show Theorem \ref{main-intro} there.

\section{Preliminaries}\label{sect-prelim}
In this preliminary section, we fix the notation and discuss various concepts related to locally convex spaces that will be used in this article.

\subsection{Notation}
Let $X$ be a lcHs (= Hausdorff locally convex space). We denote by $X^\prime$ the dual of $X$, which, unless specified otherwise, we endow with the strong topology. 

Let $X$ and $Y$ be two lcHs. We write $X \cong Y$  if $X$ and $Y$ are topologically isomorphic, and  $X \compl Y$ if $X$ is topologically isomorphic to a complemented subspace of $Y$. 

We denote by $L(X, Y)$ the space consisting of all continuous linear mappings from $X$ into $Y$. Unless specified otherwise, we endow $L(X, Y)$ with the topology of uniform convergence on bounded subsets of $X$.

We define the $\varepsilon$-product \cite{K-Ultradistributions3, S-ThDistVV} of $X$ and $Y$ as
	\[ X \varepsilon Y = L_\varepsilon(X^\prime_c, Y), \]
where the subscripts $c$ and $\varepsilon$ indicate that we endow $X^\prime$ with the topology of uniform convergence on absolutely convex compact subsets of $X$ and $L(X^\prime_c, Y)$ with the topology of uniform convergence on equicontinuous subsets of $X'$, respectively. 
If $X$ is Montel, then $X \varepsilon Y = L(X^\prime, Y)$ as lcHs.
The spaces $X \varepsilon Y$ and $Y \varepsilon X$ are canonically isomorphic via transposition \cite[p.~657]{K-Ultradistributions3}.

\subsection{Projective spectra}
A \emph{projective spectrum} $\mathscr{X}$ is a sequence of vector spaces $(X_n)_{n \in \N}$ and linear maps $\varrho^n_m : X_m \to X_n$, $n \leq m$, satisfying
	\[ \varrho^n_n = \id_{X_n} , \qquad \varrho^n_k = \varrho^n_m \circ \varrho^m_k , \qquad \forall n \leq m \leq k . \]
We write $\mathscr{X} = (X_n, \varrho^n_m)$ and call $\varrho^n_m$ the \emph{spectral mappings}. Consider the linear map
	\[ \Psi_{\mathscr{X}}  : \prod_{n \in \N} X_n \to \prod_{n \in \N} X_{n} , \quad (x_n)_{n \in \N} \mapsto (x_n - \varrho^n_{n + 1}(x_{n + 1}))_{n \in \N} . \]
We define the \emph{projective limit} of $\mathscr{X}$ as
	\[ \Proj \mathscr{X} = \ker \Psi_{\mathscr{X}} = \left\{ (x_n)_{n \in \N} \in \prod_{n \in \N} X_n \mid \varrho^n_m(x_m) = x_n \text{ for all } n \leq m \right\}.  \]
If the spectral mappings $\varrho_m^n$ are clear from the context, we also write
$$
\Proj \mathscr{X} =  \varprojlim_{n \in \N} X_n.
$$	
For $n \in \N$ we set
	\[ \varrho^n : \Proj \mathscr{X} \to X_n , \quad (x_m)_{m \in \N} \mapsto x_n. \]
Note that $\varrho^n = \varrho^n_m \circ \varrho^m$ for all $n \leq m$.
We define the  \emph{first derived projective limit} of $\mathscr{X}$ as
	\[ \Proj^1 \mathscr{X} = \left( \prod_{n \in \N} X_n \right) / \im \Psi_{\mathscr{X}} . \]

Two projective spectra $\mathscr{X} = (X_n, \varrho^n_m)$ and $\mathscr{Y} = (Y_n, \sigma^n_m)$ are called \emph{equivalent} if there are sequences $(k_n)_{n \in \N}$ and $(l_n)_{n \in \N}$ of natural numbers with $n \leq l_n \leq k_n \leq l_{n + 1}$ and linear maps $\alpha_n : X_{k_n} \to Y_{l_n}$ and $\beta_{n} : Y_{l_{n}} \to X_{k_{n-1}}$ such that $\beta_{n} \circ \alpha_{n} = \varrho^{k_{n-1}}_{k_{n}}$ and $\alpha_n \circ \beta_{n + 1} = \sigma^{l_n}_{l_{n + 1}}$. If $\mathscr{X}$ and $\mathscr{Y}$ are equivalent spectra, then $\Proj \mathscr{X} \cong \Proj \mathscr{Y}$ and $\Proj^1 \mathscr{X} \cong \Proj^1 \mathscr{Y}$ \cite[Proposition 3.1.7]{W-DerivFuncFunctAnal}.

We shall use the following fundamental property of the derived projective limit. 
\begin{proposition}\label{Proj1fprop}
Let $\mathscr{X} = (X_n, \varrho^n_m)$, $\mathscr{Y} = (Y_n, \sigma^n_m)$, and $\mathscr{Z} = (Z_n, \tau^n_m)$ be projective spectra. Let
\[
				\begin{tikzcd}
					0 \arrow{r} & X_1 \arrow{r}{f_1} & Y_1 \arrow{r}{g_1} & Z_1 \arrow{r}  & 0  \\
										0 \arrow{r} & X_2 \arrow{r}{f_2}\arrow{u}{\varrho_2^1}  & Y_2 \arrow{r}{g_2}\arrow{u}{\sigma_2^1}  & Z_2 \arrow{r}\arrow{u}{\tau_2^1}  & 0 \\
									 & \vdots \arrow{u}  & \vdots \arrow{u}    & \vdots \arrow{u}    &  
								\end{tikzcd}
			\]	
be a commutative diagram of exact sequences of vector spaces. If  $\Proj^1 \mathscr{X} =0$, then the sequence of vector spaces
$$
\SES{\Proj \mathscr{X}}{\Proj \mathscr{Y}}{\Proj \mathscr{Z}}{f}{g} 
$$
is exact, where $f((x_n)_{n \in \N}) = (f_n(x_n))_{n \in \N}$ and $g((y_n)_{n \in \N}) = (g_n(y_n))_{n \in \N}$ for  $(x_n)_{n \in \N} \in \Proj \mathscr{X}$ and $(y_n)_{n \in \N} \in \Proj \mathscr{Y}$.
\end{proposition}
A \emph{projective spectrum of lcHs} is a projective spectrum $\mathscr{X} = (X_n, \varrho^n_m)$ consisting of lcHs $X_n$ such that all the spectral mappings $\varrho^n_m$ are continuous. In this case, we always endow $\Proj \mathscr{X}$ with the projective limit topology, that is, the coarsest topology such that all the mappings $\varrho^n : \Proj \mathscr{X} \to X_n$ are continuous.

A projective spectrum $\mathscr{X} = (X_n, \varrho^n_m)$ of lcHs is called \emph{strongly reduced} if
	\[ \forall n \in \N ~ \exists m \geq n :  \varrho^n_m(X_m) \subseteq \overline{\varrho^n(\Proj \mathscr{X})} . \]

\subsection{$(\PLS)$-spaces}
 A sequence $(X_{n})_{n \in \N}$ of lcHs with $X_n \subseteq X_{n+1}$ and continuous inclusion maps is called an \emph{inductive spectrum of lcHs}. We define the \emph{inductive limit} of  $(X_{n})_{n \in \N}$ as the set  $X =\bigcup_{n \in \N} X_n$ endowed with the finest Hausdorff  locally convex topology such that all inclusion maps $X_n \to X$ are continuous (if it exists, see \cite[Lemma 24.6]{M-V-IntroFunctAnal}).  We write $\displaystyle X = \varinjlim_{n \in \N} X_{n}$.
 
 A lcHs $X$ is called an \emph{$(\LB)$-space} if  $\displaystyle X =  \varinjlim_{n \in \N} X_{n}$ for some inductive spectrum $(X_{n})_{n \in \N}$ of Banach spaces. If, in addition the inductive spectrum $(X_{n})_{n \in \N}$ can be chosen such that  the inclusion mappings $X_n \to X_{n+1}$ are compact, we say that $X$ is  an \emph{$(\LS)$-space}. A lcHs $X$ is an $(\LS)$-space if and only if it can be written as the dual of a Fr\'echet-Schwartz space \cite[Proposition 25.20]{M-V-IntroFunctAnal}.
 
A lcHs $X$ is called a   \emph{$(\PLB)$-space} if  $X =  \Proj \mathscr{X}$ for some projective spectrum $\mathscr{X}$ of $(\LB)$-spaces. Similarly,  we say that $X$ is a \emph{$(\PLS)$-space}  if  $X =  \Proj \mathscr{X}$ for some projective spectrum $\mathscr{X}$ of $(\LB)$-spaces. Every $(\PLS)$-space can be written as $X =  \Proj \mathscr{X}$ for some strongly reduced projective spectrum $\mathscr{X}$ of $(\LS)$-spaces. If $\mathscr{X}$ and $\mathscr{X}^\prime$ are two strongly reduced projective spectra of $(\LS)$-spaces with $\Proj \mathscr{X} = \Proj \mathscr{X}^\prime$, then $\mathscr{X}$ and $\mathscr{X}^\prime$ are equivalent \cite[Proposition 3.3.8]{W-DerivFuncFunctAnal}. The class of $(\PLS)$-spaces is closed under taking closed subspaces and countable products. In particular, if $X$ and $Y$ are $(\PLS)$-spaces with $X \compl Y$, then there exists a $(\PLS)$-space $Z$ such that $Y \cong X \times Z$. 

We need the following factorization result.

\begin{lemma} \label{l:factorization}
Let $\mathscr{X} = (X_n, \varrho_m^n)$ be a strongly reduced projective spectrum of $(\LS)$-spaces, let $Y$ be an $(\LS)$-space, and let $T \in L(\Proj \mathscr{X},Y)$. Then, there exist $n \in \N$ and $\widetilde{T} \in L(X_n,Y)$ such that $T = \widetilde{T} \circ \varrho^n$.
\end{lemma} 
 \begin{proof}
This follows from \cite[Proposition 3.3.8]{W-DerivFuncFunctAnal} with $\mathscr{Y}$ the constant spectrum $Y$.
\end{proof}

\subsection{Short exact sequences} A sequence
	\begin{equation}
		\label{eq:PrelimSES} \SES{X}{Y}{E}{\iota}{Q} 
	\end{equation}
of lcHs is said to be \emph{algebraically exact} if $\iota$ and $Q$ are continuous linear mappings and the sequence is exact as a sequence of vector spaces. 
If, in addition, $\iota$ is a topological embedding and $Q$ is open, we call the sequence \emph{(topologically) exact}\footnote{In the sense of category theory, this means that the sequence is exact in the category of locally convex spaces.}. The sequence \emph{splits} if $Q$ admits a continuous linear right inverse, or equivalently, if $\iota$ admits a continuous linear left inverse.

Given two lcHs $X$ and $E$, we write $\Ext^1(E, X) = 0$ if every exact sequence of the form \eqref{eq:PrelimSES} splits. Let $X$ and $E$ be $(\LS)$-spaces. We write $\Ext_{\LS}^1(E, X) = 0$ if every  exact sequence of the form \eqref{eq:PrelimSES} with $Y$ an $(\LS)$-space splits. Similarly, given two $(\PLS)$-spaces $E$ and $X$, we write  $\Ext_{\PLS}^1(E, X) = 0$ to indicate that every  exact sequence of the form \eqref{eq:PrelimSES} with $Y$ a $(\PLS)$-space splits.

We now give a useful characterization of $\Ext^1(E, X) = 0$, $\Ext_{\LS}^1(E, X) = 0$,  and $\Ext_{\PLS}^1(E, X) = 0$.  Let $X$, $Y$, and $E$ be lcHs. For $T \in L(X,Y)$, we define
	\[ T_* : L(E, X) \to L(E,Y), \quad S \mapsto T \circ S . \]

\begin{lemma}\label{charEXT}  \mbox{}
\begin{enumerate}
\item  Let $X$ and $E$ be lcHs. Then, $\Ext^1(E, X) = 0$  if and only if for every  exact sequence  
\begin{equation}
\label{EXTseq}
 \SES{X}{Y}{Z}{}{Q} 
\end{equation}
 the mapping $Q^*: L(E, Y) \to L(E,Z)$ is surjective.
\item Let $X$ and $E$ be $(\LS)$-spaces. Then, $\Ext_{\LS}^1(E, X) = 0$ if and only if for every  exact sequence \eqref{EXTseq} of $(\LS)$-spaces the mapping $Q^*: L(E, Y) \to L(E,Z)$ is surjective.
\item Let $X$ and $E$ be $(\PLS)$-spaces. Then, $\Ext_{\PLS}^1(E, X) = 0$ if and only if for every  exact sequence \eqref{EXTseq} of $(\PLS)$-spaces the mapping $Q^*: L(E, Y) \to L(E,Z)$ is surjective.
\end{enumerate}
\end{lemma}
\begin{proof}
(1) is shown \cite[Proposition 5.1.3]{W-DerivFuncFunctAnal}, while (2) and (3) follow from this result and the fact that the categories of $(\LS)$- and $(\PLS)$-spaces are quasi-abelian (see the beginning of \cite[Section 5.3]{W-DerivFuncFunctAnal}).
\end{proof}

\section{Sequence spaces}\label{sect-seqsp}
In this section, we define and discuss various types of sequence spaces.

\subsection{Fr\'echet- and $(\LB)$-type sequence spaces} By an \emph{index set} $\mathcal{I}$, we mean a non-finite countable set.
A sequence $c = (c_i)_{i \in \mathcal{I}} \in \C^{\mathcal{I}}$ of complex numbers is said to \emph{vanish at $\infty$} if for every $\varepsilon > 0$ there exists a finite subset $\mathcal{I}_\varepsilon \subseteq \mathcal{I}$ such that $|c_i| < \varepsilon$ for all $i \in \mathcal{I} \setminus \mathcal{I}_\varepsilon$.

Fix an index set $\mathcal{I}$. A sequence $w = (w_i)_{i \in \mathcal{I}}$ of positive numbers is called a \emph{weight sequence}. We associate the following Banach spaces with a weight sequence $w$:
	\[ \ell^1(w) = \{ c = (c_i)_{i \in \mathcal{I}} \in \C^{\mathcal{I}} \mid \|c\|_{1, w} = \sum_{i \in \mathcal{I}} |c_i| w_i < \infty \} \]
and
	\[ \ell^\infty(w) = \{ c = (c_i)_{i \in \mathcal{I}} \in \C^{\mathcal{I}} \mid \|c\|_{\infty, w} = \sup_{i \in \mathcal{I}} |c_i| w_i < \infty \} . \]

A \emph{weight matrix} $a = (a_n)_{n \in \N}$ is a sequence of weight sequences $a_n$ such that $a_{n, i} \leq a_{n + 1, i}$ for all  $i \in \mathcal{I}$ and $n \in \N$. For $p \in \{1, \infty\}$ and a weight matrix $a$, we define the Fr\'echet space
	\[ \lambda^p(a) = \varprojlim_{n \in \N} \ell^p(a_n). \]

A \emph{dual weight matrix} $b = (b_N)_{N \in \N}$ is a sequence of weight sequences $b_N$ such that $b_{N + 1, i} \leq b_{N, i}$ for all  $i \in \mathcal{I}$ and $N \in \N$. 
For $p \in \{1, \infty\}$ and a dual weight matrix $b$, we define the $(\LB)$-space
	\[ k^p(b) = \varinjlim_{N \in \N} \ell^p(b_N). \]

Given a  (dual) weight matrix $a$, we define  $a^\circ = (1/a_N)_{N \in \N}$. 

A weight sequence $a$ is said to satisfy condition (S) if 
$$
	\forall n \in \N ~ \exists m \geq n :  \frac{a_{n}}{a_{m}} \mbox{ vanishes at $\infty$}. 
$$
In such a case, for $p \in \{1,\infty\}$, $\lambda^p(a)$ is a Fr\'echet-Schwartz space and  $k^p(a)$  is an $(\LS)$-space. Moreover, $(\lambda^p(a))^\prime = k^q(a^\circ)$ and $(k^p(a^\circ))^\prime = \lambda^q(a)$ as lcHs, where $q$ is the conjugate index of $p$.

We now recall a standard result from the splitting theory of Fr\'echet spaces \cite{V-Ext1Frechet}. A weight matrix $a$ is said to satisfy condition $(\DN)$ if
	\[ \exists n \in \N ~ \forall m \geq n, \theta \in (0, 1) ~ \exists k \geq m, C > 0 ~ \forall i \in \mathcal{I} : a_{m, i} \leq C a_{k, i}^{\theta} a_{n, i}^{1 - \theta}, \]
while $a$ is said to satisfy $(\Omega)$ if
	\[ \forall n \in \N ~ \exists m \geq n ~ \forall k \geq m ~ \exists \theta \in (0, 1) , C > 0 ~ \forall i \in \mathcal{I} :  a_{k, i}^{\theta} a_{n, i}^{1 - \theta} \leq C a_{m, i}. \]
	\begin{remark}\label{rem:DN-Omega}
	The above conditions are closely related to the properties $(\DN)$ and $(\Omega)$ for Fr\'echet spaces \cite{M-V-IntroFunctAnal, V-Ext1Frechet}. More precisely, a weight matrix $a$ satisfies  $(\DN)$ or $(\Omega)$ if and only if the Fr\'echet space $\lambda^{p}(a)$, $p \in \{1,\infty\}$, does so.	\end{remark}

	\begin{proposition}
		\label{p:SplittingWeightMatrices}
		Let $v$ and $w$ be weight matrices satisfying  \emph{(S)} such that $v$ satisfies $(\DN)$ and $w$ satisfies $(\Omega)$. Then,
		\begin{enumerate}
		\item $\Ext^1(\lambda^1(v), \lambda^\infty(w)) = 0$. 
		\item $\ExtLS^1(k^1(w^\circ), k^\infty(v^\circ)) = 0$.
		\end{enumerate}
	\end{proposition}
	
	\begin{proof}
In view of Remark \ref{rem:DN-Omega}, part (1) follows from \cite[Theorem 5.1]{V-Ext1Frechet}. For (2), let 
$$
 \SES{ k^\infty(v^\circ)}{Y}{k^1(w^\circ)}{\iota}{Q} 
$$
be an  exact sequence of $(\LS)$-spaces. By 	\cite[Proposition 26.24]{M-V-IntroFunctAnal}, its dual sequence 
$$
 \SES{ \lambda^\infty(w)}{Y'}{\lambda^1(v)}{Q^t}{\iota^t} 
$$
is also exact. Part (1) implies that this sequence splits. Let $R \in L(\lambda^1(v),Y')$ be a right inverse of $\iota^t$. Since $(\LS)$-spaces are reflexive, it holds that $Y'' = Y$. Hence, $R^t \in  L(Y,k^\infty(v^\circ))$ is a left inverse of $\iota$.
	\end{proof}
	
By an \emph{exponent sequence $\alpha = (\alpha_i)_{i \in \N}$}, we mean a non-decreasing sequence of positive numbers such that $\lim_{i \to \infty} \alpha_i = \infty$. Given an exponent sequence $\alpha$,  we define the weight matrices $v^\infty_\alpha = ((e^{n \alpha_i})_{i \in \N})_{n \in \N}$ and  $v^0_\alpha = ((e^{-\alpha_i/n})_{i \in \N})_{n \in \N}$. Note that $v^\infty_\alpha$ satisfies (S), $(\DN)$, and $(\Omega)$, while  $v^0_\alpha$ satisfies (S) and $(\Omega)$, but not $(\DN)$. For $p \in \{1, \infty\}$, we define 
\[ \Lambda^p_\infty(\alpha) = \lambda^p(v^\infty_\alpha), \qquad  \Lambda^p_0(\alpha) = \lambda^p(v^0_\alpha) . \]
We call $\Lambda^p_\infty(\alpha)$ and $\Lambda^p_0(\alpha)$ power series spaces of infinite and finite type, respectively.

An exponent sequence $\alpha$ is called \emph{stable} if
	\[ \sup_{i \in \N} \frac{\alpha_{2i}}{\alpha_i} < \infty. \]
The sequence $\alpha$ is stable if and only if  $\Lambda^p_r(\alpha) \cong  \Lambda^p_r(\alpha) \times  \Lambda^p_r(\alpha)$ for $p \in \{1,\infty\}$ and $r \in \{0,\infty\}$. We say that $\alpha$ satisfies condition (N) if
	\[ \sup_{i \in \N} \frac{\log (1 + i)}{\alpha_i} < \infty . \]
Then, $\alpha$ satisfies (N) if and only if the Fr\'echet space $\Lambda^p_\infty(\alpha)$, $p \in \{1,\infty\}$, is nuclear \cite[Proposition 29.6(1)]{M-V-IntroFunctAnal}. In such a case, $\Lambda^1_\infty(\alpha) = \Lambda^\infty_\infty(\alpha)$ and we simply denote this space by $\Lambda_\infty(\alpha)$.

\subsection{$(\PLB)$-type sequence spaces} 
Let $\mathcal{I}$ be an index set.
A \emph{weight matrix system} $\mathcal{A} = (a_n)_{n \in \N}$ is a sequence of dual weight matrices $a_n = (a_{n, N})_{N \in \N}$ such that $a_{n, N + 1, i} \leq a_{n, N, i} \leq a_{n + 1, N, i}$ for all $i \in \mathcal{I}$ and $n, N \in \N$.
For $p \in \{1, \infty\}$ and a weight matrix system $\mathcal{A}$, we define the $(\PLB)$-space
	\[ \lambda^p(\mathcal{A}) = \varprojlim_{n \in \N} k^p(a_n) . \]
Given two weight matrices $v$ and $w$ on index sets $\mathcal{I}$ and $\mathcal{J}$, respectively, we define the following weight matrix system on $\mathcal{I} \times \mathcal{J}$:
	\[ \mathcal{A}_{v, w} = (((v_{n, i} / w_{N, j})_{(i, j) \in \mathcal{I} \times \mathcal{J}})_{N \in \N})_{n \in \N} . \]
We write $\lambda^p(v, w) = \lambda^p(\mathcal{A}_{v, w})$ for $p \in \{1, \infty\}$.
In the following lemma, we study the basic linear topological structure of the spaces $\lambda^p(v,w)$, $p \in \{1, \infty\}$.	
	\begin{lemma}
		\label{l:EpsilonProdWeights}
Let $v$ and $w$ be weight matrices  on the index sets $\mathcal{I}$ and $\mathcal{J}$, respectively, satisfying  \emph{(S)} and let $p \in \{1, \infty\}$. Then, there exist dual weight matrices $a_n$, $n \in \N$, on  $\mathcal{I} \times \mathcal{J}$  such that 
\begin{itemize}
\item $a_n^\circ$ satisfies  \emph{(S)} for each $n \in \N$.
\item  $k^p(a_{n+1}) \subseteq k^p(a_n)$ continuously for all $n \in \N$ and the projective spectrum $(k^p(a_n))$ (with inclusion as spectral mappings) is strongly reduced.
\item  If $w$ satisfies $(\DN)$ or $(\Omega)$, then each $a^\circ_n$ does so as well.
\item  $\displaystyle \lambda^p(v,w) =  \varprojlim_{n \in \N} k^p(a_n)$ as lcHs. 
\end{itemize}
 In particular, $ \lambda^p(v,w)$ is a $(\PLS)$-space.
 \end{lemma}
 \begin{proof}
 Since $v$ satisfies (S), we may assume without loss of generality that $v_n/v_{n+1}$ vanishes at $\infty$ for all $n \in \N$. 
Define
			\[ \alpha_{n, i} = \sqrt{\max\{1, \log ( v_{n +1, i} / v_{n, i})\}} , \qquad i \in \mathcal{I}, \, n \in \N . \] 
Then, for each $n \in \N$, $\alpha_n = (\alpha_{n,i})_{i \in \mathcal{I}}$ is a sequence of positive numbers tending to $\infty$ (meaning that $(1/\alpha_{n,i})_{i \in \mathcal{I}}$ vanishes at $\infty$) such that for all $N \in \N$ there is a $c >0$ with
$$
c v_{n,i} \leq v_{n+1,i}e^{-N\alpha_{n,i}}\leq v_{n+1,i}, \qquad \forall i \in \mathcal{I}. 
$$
Define 
$$
a_{n, N, (i, j)} = \frac{v_{n, i}}{w_{N, j}} e^{-N \alpha_{n, i}} , \qquad (i, j) \in \mathcal{I} \times \mathcal{J} , n,N \in \N. 
$$
Then one can readily verify that the dual weight matrices $a_n = (a_{n,N})_{N \in \N}$, $n \in \N$, verify all requirements.
	\end{proof}

We will need the following $\varepsilon$-product representation of the space $\lambda^\infty(v,w)$. 
\begin{lemma}\label{lemma:tensor}
Let $v$ and $w$ be weight matrices satisfying \emph{(S)}. Then,
$$
\lambda^\infty(v,w) =  \lambda^\infty(v) \varepsilon  k^\infty(w^\circ)
$$
as lcHs.
\end{lemma}
\begin{proof}
Since $\lambda^\infty(v)$ satisfies the weak approximation property (as it satisfies (S)), it holds that $\lambda^\infty(v) \varepsilon  k^\infty(w^\circ) = \lambda^\infty(v)  \widehat{\otimes}_\varepsilon  k^\infty(w^\circ)$ \cite[Proposition 1.4]{K-Ultradistributions3}, where the latter space denotes the completion of the tensor product  $\lambda^\infty(v)  \otimes  k^\infty(w^\circ)$ with respect to the topology inherited from $\lambda^\infty(v) \varepsilon  k^\infty(w^\circ)$.  Hence, the result follows from \cite[Proposition 10]{W-WeighPLBSpContFuncTensProd}.
	\end{proof}

Given two exponent sequences $\alpha$ and $\beta$, we define for $p \in \{1,\infty\}$ and $r \in \{0,\infty\}$ the following $(\PLS)$-type power series spaces
	\begin{equation}
	\label{PSPLSdef} \Lambda^p_r(\alpha, \beta) = \lambda^p(v^r_\alpha, v^r_\beta). 
\end{equation}
If $\alpha$ and $\beta$ satisfy (N), then $\Lambda^\infty_\infty(\alpha, \beta) =  \Lambda^1_\infty(\alpha, \beta)$ and we simply denote this space by $\Lambda_\infty(\alpha, \beta)$.
	\begin{remark}
	\label{PSPLS}
	We note that the notation \eqref{PSPLSdef} is not standard. Let $\alpha = (\alpha_i)_{i \in \N}$ and $\beta = (\beta_i)_{i \in \N}$  be sequences of non-negative numbers with $\alpha_i + \beta_i \to \infty$. Consider the space
	$$
	\widetilde{\Lambda}^\infty_\infty(\alpha, \beta) = \varprojlim_{n \in \N}  \varinjlim_{N \in \N} X_{n,N},
	$$
	where 
	$$
	X_{n,N} = \{ c=(c_i)_{i \in \N} \mid \| c\|_{n,N}  = \sup_{i \in \N} |c_i| e^{n \alpha_i - N \beta_i} < \infty\}.
	$$
 The spaces $\widetilde{\Lambda}^1_\infty(\alpha, \beta), \widetilde{\Lambda}^\infty_0(\alpha, \beta)$, and  $\widetilde{\Lambda}^1_0(\alpha, \beta)$ are defined similarly. The space  $\widetilde{\Lambda}^p_r(\alpha, \beta)$, $p \in \{1,\infty\}$ and $r \in \{0,\infty\}$,  is commonly denoted by  $\Lambda^p_r(\alpha, \beta)$, see e.g. \cite{B-D-ParamDepSolDiffEqSpDistSplittingSES, Vogt-Topics}.	
	\end{remark}
	
\section{A splitting result for $(\PLS)$-type sequence spaces}\label{sect-splitting}

In this section, we show the following splitting result for $(\PLS)$-type sequence spaces; it generalizes Proposition \ref{p:SplittingWeightMatrices}.

	\begin{theorem}
		\label{t:SplittingResult}
		Let $v, w, x, y$ be weight matrices satisfying \emph{(S)}.  If $v, y$ satisfy $(\DN)$ and $w, x$ satisfy $(\Omega)$, then 
		$$
		\ExtPLS^1(\lambda^1(v, w), \lambda^\infty(x, y)) = 0.
		$$
	\end{theorem}
	In the next section, we will use the following particular case of Theorem \ref{t:SplittingResult}.	
	\begin{corollary}
		\label{c:SplittingResultPS}
		Let $\alpha,\beta,\gamma,\delta$ be exponent sequences.  Then, 
		$$
		\ExtPLS^1(\Lambda^1_\infty(\alpha, \beta),\Lambda^\infty_\infty(\gamma, \delta)) = 0.
		$$
	\end{corollary}
	
The remainder of this section is devoted to the proof of Theorem \ref{t:SplittingResult}. In his PhD thesis \cite[Theorem 5.14]{K-SplitPowerSeriesSpPLS}, Kunkle showed a result analogous to Corollary \ref{c:SplittingResultPS}, namely,
\begin{equation}
\label{resultKun}
\ExtPLS^1(\widetilde{\Lambda}^1_\infty(\alpha, \beta), \widetilde{\Lambda}^\infty_\infty(\gamma, \delta)) = 0;
\end{equation}
see Remark \ref{PSPLS} for the definition of the spaces $\widetilde{\Lambda}^1(\alpha, \beta)$ and  $\widetilde{\Lambda}^\infty(\gamma, \delta)$. We closely follow Kunkle’s approach to show Theorem  \ref{t:SplittingResult}, and we emphasize that all key ideas in this section are due to him; we have only made some technical modifications in the final part of the argument.

As in the case of Fr\'echet spaces \cite{V-Ext1Frechet, W-DerivFuncFunctAnal}, the proof of Theorem \ref{t:SplittingResult} consists of three steps. First, we show that, under suitable conditions on two $(\PLS)$-spaces $E$ and $X$, $\ExtPLS^1(E,X) =0$ follows from the vanishing of the derived projective limit of an appropriate projective spectrum associated with $E$ and $X$. Second, we establish a sufficient condition for the latter by applying a result from \cite{F-K-W-ProjLimFuncSpectraWebbed} (see also \cite{K-SplitPowerSeriesSpPLS}). Finally, we verify that the assumptions of the previous two steps are satisfied for the spaces appearing in Theorem~\ref{t:SplittingResult},  using that they are sequence spaces.

\subsection{Locally splitting spectra}
Let $E$ be a lcHs and let $\mathscr{X} = (X_{n}, \varrho^{n}_{m})$ be a projective spectrum of $(\LS)$-spaces. We denote by $\mathscr{L}(E, \mathscr{X})$ the projective spectrum $(L(E, X_n), (\varrho^n_m)_*)$.
Note that
	\[ L(E, \Proj \mathscr{X}) \to \Proj \mathscr{L}(E, \mathscr{X}) , \quad T \mapsto ((\varrho^n)_*(T))  \]
is an isomorphism of vector spaces.
If $\mathscr{X}$ and $\mathscr{X}^\prime$ are equivalent projective spectra of $(\LS)$-spaces, then the spectra $\mathscr{L}(E, \mathscr{X})$ and $\mathscr{L}(E, \mathscr{X}^\prime)$ are also equivalent.

Let $X = \Proj \mathscr{X}$ be a $(\PLS)$-space, with $\mathscr{X}$  a strongly reduced projective spectrum of $(\LS)$-spaces.
We write $\Proj^1 L(E, X) = 0$ to indicate that $\Proj^1 \mathscr{L}(E, \mathscr{X}) = 0$.
Since all strongly reduced projective spectra of $(\LS)$-spaces defining $X$ are equivalent, this definition does not depend on the choice of the strongly reduced spectrum of $(\LS)$-spaces defining $X$.

		Let $\mathscr{E} = (E_{n}, \sigma^{n}_{m})$ and $\mathscr{X} = (X_{n}, \varrho^{n}_{m})$ be two projective spectra of $(\LS)$-spaces. We call the pair $(\mathscr{E}, \mathscr{X})$  \emph{locally splitting} if for all $n,m \in \N$
			\[ \Ext^{1}_{\LS}(E_{m}, X_{n}) = 0. \]
		A pair of $(\PLS)$-spaces $(E, X)$ is called \emph{locally splitting} if there exist strongly reduced  spectra $\mathscr{E}$ and $\mathscr{X}$  of $(\LS)$-spaces such that $E = \Proj \mathscr{E}$, $X = \Proj \mathscr{X}$, and $(\mathscr{E}, \mathscr{X})$ is locally splitting.
	
	\begin{proposition}
		\label{p:Proj^1=0=>Ext^1=0}
		Let $E$ and $X$ be $(\PLS)$-spaces such that $(E, X)$ is locally splitting. If $\Proj^{1} L(E, X) = 0$, then $\ExtPLS^1(E,X) =0$.
		 	\end{proposition}
	
	\begin{proof}
	By Lemma \ref{charEXT}(3), it suffices to show that for every  exact sequence 
		$$
			 \SES{X}{Y}{Z}{\iota}{Q}, 
		$$
of $(\PLS)$-spaces, the mapping $Q_*: L(E, Y) \to L(E, Z)$ is surjective. Take  two strongly reduced projective spectra $\mathscr{E} = (E_n, \sigma^n_m)$ and $\mathscr{X} = (X_n, \varrho^n_m)$ of $(\LS)$-spaces with $E = \Proj \mathscr{E}$, $X = \Proj \mathscr{X}$, and such that $(\mathscr{E}, \mathscr{X})$ is locally splitting. By \cite[Proposition 5.3.2]{W-DerivFuncFunctAnal}, there exist strongly reduced projective spectra $\mathscr{Y} = (Y_n, \kappa^n_m)$ and $\mathscr{Z} = (Z_n, \tau^n_m)$ of $(\LS)$-spaces with $Y = \Proj \mathscr{Y}$ and $Z = \Proj \mathscr{Z}$, a strictly increasing sequence $(k_n)_{n \in \N}$ of natural  numbers, and a commutative diagram of  exact sequences 
			\begin{equation}
			\label{intermexact}
				\begin{tikzcd}
					0 \arrow{r} & X_{k_1} \arrow{r}{\iota_1} & Y_1 \arrow{r}{Q_1} & Z_1 \arrow{r}  & 0  \\
										0 \arrow{r} & X_{k_2} \arrow{r}{\iota_2}\arrow{u}{\varrho_{k_2}^{k_1}}  & Y_2 \arrow{r}{Q_2}\arrow{u}{\kappa_2^1}  & Z_2 \arrow{r}\arrow{u}{\tau_2^1}  & 0 \\
									 & \vdots \arrow{u}  & \vdots \arrow{u}    & \vdots \arrow{u}    &  
								\end{tikzcd}
			 \end{equation}
		such that $\kappa^n \circ \iota = \iota_n \circ \varrho^{k_n}$ and $\tau^n \circ Q = Q_n \circ \kappa^n$ for all $n \in \N$.  We claim that for each $n \in \N$, the sequence of vector spaces
					\begin{equation}
					\label{intermexact}
					 \SES{L(E, X_{k_n})}{L(E, Y_{n})}{L(E, Z_{n})}{(\iota_n)_*}{(Q_n)_{*}} 
					 \end{equation}
					is exact. Before we prove the claim, let us show how it implies the result.
					From \eqref{intermexact}, we would obtain the  following commutative diagram of exact sequences of vector spaces
			\[
				\begin{tikzcd}
					0 \arrow{r} & L(E,X_{k_1}) \arrow{r}{(\iota_1)_*} &L(E,Y_1) \arrow{r}{(Q_1)_*} & L(E,Z_1) \arrow{r} &0  \\
										0 \arrow{r} & L(E,X_{k_2}) \arrow{r}{(\iota_2)_*}\arrow{u}{(\varrho_{k_2}^{k_1})_*}  & L(E,Y_2) \arrow{r}{(Q_2)_*}\arrow{u}{(\kappa_2^1)_*}  & L(E,Z_2) \arrow{r}\arrow{u}{(\tau_2^1)_*} &0  \\
									 & \vdots \arrow{u}  & \vdots \arrow{u}    & \vdots \arrow{u}    &  
								\end{tikzcd}
			 \]
Define $\widetilde{\mathscr{X}} = (X_{k_n}, \varrho^{k_n}_{k_m})$. Since
		$$
		\Proj^{1} \mathscr{L}(E, \widetilde{\mathscr{X}}) = \Proj^{1} \mathscr{L}(E, \mathscr{X}) =  0,
		$$
		the result follows from Proposition \ref{Proj1fprop}. We now show that the sequence  \eqref{intermexact}  is exact. It is clear that $(\iota_n)_*$ is injective and that $\im (\iota_n)_* = \ker (Q_n)_*$. Let us prove that the mapping $(Q_n)_*$ is surjective. Take $S \in L(E, Z_{n})$ arbitrary. By Lemma \ref{l:factorization}, there are $m \in \N$ and $\widetilde{S} \in L(E_m, Z_n)$ such that $\widetilde{S} \circ \sigma^m = S$. Since $\Ext^{1}_{\LS}(E_{m}, X_{k_n}) = 0$ and the sequence 
$$
					 \SES{X_{k_n}}{Y_{n}}{Z_{n}}{(\iota_n)_*}{(Q_n)_{*}} 
$$
 is  exact, Lemma \ref{charEXT}(2) implies that the mapping $(Q_n)_*: L(E_m, Y_n) \to L(E_m, Z_n)$ is surjective. Hence, there is $\widetilde{T} \in L(E_m, Y_n)$ with $(Q_n)_*(\widetilde{T}) = \widetilde{S}$. Set $T = \widetilde{T} \circ \sigma^m \in L(E, Y_n)$. Then, $(Q_n)_*(T) = S$.	
	\end{proof}
	
\subsection{A sufficient condition for the vanishing of $\Proj^{1} L(E, X)$}

Let $X$ be a lcHs. A non-increasing sequence $(A_r)_{r \in \N}$ of subsets in $X$ is called \emph{completing} if there exists a sequence $(\lambda_r)_{r \in \N}$ of positive numbers such that for every sequence $(x_r)_{r \in \N} \in \prod_{r \in \N} A_r$ the series $\sum_{r \in \N} \lambda_n x_r$ converges in $X$.

	\begin{proposition}[{\cite[Proposition 3.3]{F-K-W-ProjLimFuncSpectraWebbed}}]
		\label{p:Proj1=0Webbed}
		Let $\mathscr{X} = (X_{n}, \varrho^{m}_{n})$ be a projective spectrum of lcHs such that the spaces $X_n$ have completing sequences $(A^n_{r})_{r \in \N}$ with $\varrho^{n}_{n + 1} A^{n + 1}_r \subseteq A^{n}_r$ for all $n, r \in \N$.
		If
			\[
				\forall n \in \N ~ \exists m \geq n ~ \forall k \geq m: \varrho^{n}_{m} X_m \subseteq \varrho^{n}_{k} X_k + A^n_{n},
			\]
		then $\Proj^{1} \mathscr{X} = 0$.
	\end{proposition}

Let $r \in \N \cup \{\infty\}$. We define
			\[ \mathfrak{C}^{r} = \{ \alpha = (\alpha_1, \alpha_2, \alpha_3) \mid \alpha_1 \in \N , \alpha_2 : \{ 0, \ldots, r \} \to \N , \alpha_3 : \{ 0, \ldots, r \} \to [1, \infty) \}. \]
		For $r \leq s \leq \infty$ and $\alpha \in  \mathfrak{C}^{s}$ we denote
		$$
		\alpha_{\mid r} = (\alpha_1, (\alpha_2)_{\mid \{0, \ldots, r\}}, (\alpha_3)_{\mid \{0, \ldots, r\}}) \in   \mathfrak{C}^{r}.
		$$
		We set
			\[ \mathfrak{C} = \bigsqcup_{r \in \N \cup \{\infty\}} \mathfrak{C}^r. \]
		For $\alpha \in \mathfrak{C}$ we write $\ell(\alpha) = r$ if $\alpha \in \mathfrak{C}^r$.	
	\begin{proposition}
		\label{p:Proj1L(X, Y)=0}
Let $\mathscr{X} = (X_{n}, \varrho^{m}_{n})$  and $\mathscr{E} = (E_n, \sigma^n_m)$ be  projective spectra of $(\LS)$-spaces such that  $\mathscr{E}$ is strongly reduced. 
Let  $\displaystyle X_n = \varinjlim_{N \in \N} X_{n,N}$ and $\displaystyle E_n = \varinjlim_{N \in \N} E_{n,N}$, with $(X_{n,N})_{N \in \N}$ and $(E_{n,N})_{N \in \N}$ inductive spectra of Banach spaces.
Set $E = \Proj  \mathscr{E}$. For $n \in \N$ and  $\alpha \in \mathfrak{C}$, we define $A^n_\alpha$ as the set consisting of all $T \in L(E, X_n)$ for which there exists $\widetilde{T} \in L(E_{\alpha_1},X_n)$ with $\widetilde{T}  \circ \varrho^{\alpha_1}  = T$ such that
$$
 \sup_{\|x\|_{E_{\alpha_1, M}} \leq 1} \|\widetilde{T}(x)\|_{X_{n,\alpha_2(M)}} \leq \alpha_3(M), \qquad \forall  M \leq \ell(\alpha).
$$
If
			\begin{equation}
			\label{condmg}
				\begin{gathered}
			\exists \alpha \in \mathfrak{C}^{\infty} ~ \forall n \in \N ~ \exists m \in \N ~ \forall k \in \N, \beta \in \mathfrak{C}^n ~ \exists \gamma \in \mathfrak{C}^n:\\
								(\varrho^n_m)_{*}(A^m_\beta) \subseteq (\varrho^n_k)_{*}(A^k_\gamma) + A^n_{\alpha\mid n},
				\end{gathered}
			\end{equation}
					then $\Proj^1\mathscr{L}(E, \mathscr{X}) =0$.
	\end{proposition}

	\begin{proof}
By Lemma \ref{l:factorization} and Grothendieck's factorization theorem, it holds that for all $n,m \in \N$
	$$
	L(E,X_m) = \bigcup_{\beta \in \mathfrak{C}^n}A^m_\beta.
	$$
	Hence,  \eqref{condmg} implies that
	\[
		\begin{gathered}
					\exists \alpha \in \mathfrak{C}^\infty ~ \forall n \in \N ~ \exists m \geq n ~ \forall k \geq m : \\
					(\varrho^n_m)_{*}(L(E, X_m)) \subseteq (\varrho^n_k)_{*}(L(E, X_k)) + A^n_{\alpha\mid n}.
				\end{gathered}
			\]
		In view of Proposition \ref{p:Proj1=0Webbed}, it suffices to show that, for each $\alpha \in \mathfrak{C}^\infty$, the sequence $(A^n_{\alpha\mid r})_{r \in \N}$ is completing in $L(E, X_n)$.
		Let $(T_r)_{r \in \N} \in \prod_{r \in \N} A^n_{\alpha\mid r}$ be arbitrary. For every $r \in \N$ there exists $\widetilde{T}_r \in L(E_{\alpha_1},X_n)$ with $\widetilde{T}_r  \circ \varrho^{\alpha_1}  = T_r$ such that for all
$$
 \sup_{\|x\|_{E_{\alpha_1, M}} \leq 1} \|\widetilde{T}_r(x)\|_{X_{n,\alpha_2(M)}} \leq \alpha_3(M), \qquad \forall M \leq r.
$$
Fix $M \in \N$. For all $x \in E_{\alpha_1,M}$  it holds that
			\[ \sum_{r \geq M} 2^{-r}   \|\widetilde{T}_{r}(x) \|_{X_{n, \alpha_2(M)}}\leq 2 \|x\|_{E_{\alpha_1, M}} \alpha_3(M) . \]
This implies that $\sum_{r \in \N} 2^{-r} \widetilde{T}_{r}$ converges to an element of $L(E_{\alpha_1}, X_n)$. Hence, 
$$
\sum_{r \in \N} 2^{-r} T_{r} = \left(\sum_{r \in \N} 2^{-r} \widetilde{T}_r \right) \circ \varrho^{\alpha_1}
$$ 
converges to an element of $L(E, X_n)$.
	\end{proof}

\subsection{The proof of Theorem \ref{t:SplittingResult}}	

We now apply Propositions  \ref{p:Proj^1=0=>Ext^1=0} and \ref{p:Proj1L(X, Y)=0}  to show  Theorem \ref{t:SplittingResult}. We start by discussing when the pair $(\lambda^1(v,w), \lambda^\infty(x,y)$ is locally splitting.

		\begin{lemma}
	\label{l:SeqSpLocallySplitting}
		Let $v, w, x, y$ be weight matrices satisfying \emph{(S)}. If $y$ satisfies $(\DN)$ and $w$ satisfies $(\Omega)$, then the pair $(\lambda^1(v, w), \lambda^\infty(x, y))$ is locally splitting.
	\end{lemma}
	\begin{proof}
	This follows from Proposition \ref{p:SplittingWeightMatrices}(2) and Lemma \ref{l:EpsilonProdWeights}.
	\end{proof}

Next, we study when $\Proj^1(\lambda^1(v, w), \lambda^\infty(x, y)=0$ with the aid of Proposition \ref{p:Proj1L(X, Y)=0}.  Let $\mathcal{A} = (a_n)_{n \in \N}$  be a weight matrix system on an index set  $\mathcal{I}$ and let $b$ be a dual weight matrix on an index set $\mathcal{J}$. For $\alpha \in \mathfrak{C}$ we define $A_\alpha(\mathcal{A},b)$ as the set consisting of all $T \in L(\lambda^1(\mathcal{A}), k^\infty(b))$ for which there exists $\widetilde{T} \in L( k^1(a_{\alpha_1}), k^\infty(b))$ with $\widetilde{T}_{\mid \lambda^1(\mathcal{A})}  = T$ such that 
$$
 \sup_{\|c\|_{1, a_{\alpha_1,M}} \leq 1} \|\widetilde{T}(c)\|_{\infty,b_{\alpha_2(M)}} \leq \alpha_3(M), \qquad \forall M \leq \ell(\alpha) .
$$

		\begin{lemma}
		\label{l:CharBoundedSp}
		Let $\mathcal{A} = (a_n)_{n \in \N}$  be a weight matrix system on an index set $\mathcal{I}$, let $b$ be a dual weight matrix on an index set $\mathcal{J}$, and let $\alpha \in \mathfrak{C}^\infty$. 
		\begin{enumerate}
		\item   Let $T_{i, j} \in \C$, $(i, j) \in \mathcal{I} \times \mathcal{J}$, be numbers satisfying
		\begin{equation}
\label{coochar2} |T_{i, j}| \leq \inf_{M  \in \N} \alpha_3(M) \frac{a_{\alpha_1, M, i}}{b_{\alpha_2(M), j}}, \qquad \forall (i, j) \in \mathcal{I} \times \mathcal{J} . 
\end{equation}
	Then,
	$$
T: \lambda^1(\mathcal{A}) \to k^\infty(b), \quad c \mapsto  \left(\sum_{i \in \mathcal{I}} c_i T_{i, j}\right)_{j \in \mathcal{J}},
	$$
	is a well-defined mapping  that belongs to  $A_\alpha(\mathcal{A},b)$.
	\item  If $T \in A_\alpha(\mathcal{A},b)$, there exist unique numbers  $T_{i, j} \in \C$, $(i, j) \in \mathcal{I} \times \mathcal{J}$, satisfying \eqref{coochar2} such that 
	$$
				T(c) = \left(\sum_{i \in \mathcal{I}}  T_{i, j}c_i\right)_{j \in \mathcal{J}} , \qquad \forall c \in \lambda^1(\mathcal{A}) . 
$$
			\end{enumerate}
	\end{lemma}
	\begin{proof}
(1) For all $M \in \N$ and  $c \in \ell^1( a_{\alpha_1,M})$, it holds that
$$
\sup_{j \in \mathcal{J}} b_{\alpha_2(M), j} \sum_{i \in \mathcal{I}}  |T_{i, j}||c_i| \leq  \alpha_3(M)\|c\|_{1, a_{\alpha_1,M}}.
$$
This implies the result. \\
(2) Let $c_{00}(\mathcal{I})$ denote the space consisting of all $c \in \C^\mathcal{I}$ with only finitely  many non-zero entries.	 For $i \in \mathcal{I}$ we set $e_i = (\delta_{i,i'})_{i' \in \mathcal{I}}$, where $\delta_{i,i'}$ is the Kronecker delta symbol.  
Put $T(e_i) = (T_{i,j})_{j \in \mathcal{J}}$. Then,
\begin{equation}
\label{reprfins}
T(c) = \left(\sum_{i \in \mathcal{I}}  T_{i, j}c_i\right)_{j \in \mathcal{J}}, \qquad \forall c \in c_{00}(I).
\end{equation}
Since $T \in A_\alpha(\mathcal{A},b)$, we find that for all $M \in \N$ and  $(i, j) \in \mathcal{I} \times \mathcal{J}$
$$
|T_{i,j}| \leq \frac{ \|T(e_i)\|_{\infty,b_{\alpha_2(M)}}}{b_{\alpha_2(M), j}} \leq \alpha_3(M) \frac{\|e_i\|_{1, a_{\alpha_1,M}}}{b_{\alpha_2(M), j}} = \alpha_3(M) \frac{a_{\alpha_1,M}}{b_{\alpha_2(M), j}}.
$$
The result now  follows from \eqref{reprfins}, part (1), and the fact that  $c_{00}(\mathcal{I})$ is dense in  $\lambda^1(\mathcal{A})$.
\end{proof}

From now on, we fix four weight matrices $v, w, x,y$ on index sets $\mathcal{I}_0$, $\mathcal{I}_1$, $\mathcal{J}_0$, $\mathcal{J}_1$, respectively. We define the index sets $\mathcal{I} = \mathcal{I}_0 \times \mathcal{I}_1$ and $\mathcal{J} = \mathcal{J}_0 \times \mathcal{J}_1$,  and write $\mathcal{A} = \mathcal{A}_{v,w}$ and $\mathcal{B} = \mathcal{A}_{x,y}$.  Thus,
$$
a_{n,N,i} = \frac{v_{n,i_0}}{w_{N,i_1}}, \quad i = (i_0,i_1) \in \mathcal{I}, \qquad b_{n,N,j} = \frac{x_{n,j_0}}{y_{N,j_1}}, \quad j= (j_0,j_1) \in \mathcal{J}.
$$
For $n \in \N$, $\alpha \in \mathfrak{C}$, and  $(i, j) \in \mathcal{I} \times \mathcal{J}$, we define
$$
		C^n_{\alpha}(i, j) = \inf_{M \leq \ell(\alpha)} \alpha_3(M) \frac{a_{\alpha_1, M, i}}{b_{n,\alpha_2(M), j}}  =\frac{v_{\alpha_1, i_0}}{x_{n, j_0}} \inf_{M \leq \ell(\alpha)} \alpha_3(M) \frac{y_{\alpha_2(M), j_1}}{w_{M, i_1}}.  
$$

	\begin{lemma}
		\label{l:CharOfDecomposition} Let $k,m,n \in \N$ and $\alpha, \beta, \gamma \in \mathfrak{C}$ with $\ell(\beta) \leq \min\{\ell(\alpha), \ell(\gamma)\}$. Assume that $y$ satisfies $(\DN)$. 
		If
						\begin{equation}
				\label{eq:DecompositionCond} 
				C^m_\beta(i, j) \leq \min \{ C^k_\gamma(i, j),C^n_\alpha(i, j)\} , \qquad \forall (i, j) \in \mathcal{I} \times \mathcal{J}, 
			\end{equation}
			then
			\begin{equation}
				\label{eq:Decomposition} 
				A_\beta(\mathcal{A}, b_m) \subseteq A_\gamma(\mathcal{A}, b_k) +  A_\alpha(\mathcal{A}, b_n).
			\end{equation}
			\end{lemma}

	\begin{proof}
	Define
			\[ \mathcal{Q} = \{ (i, j) \in \mathcal{I} \times \mathcal{J} \mid C^m_\beta(i, j) \leq C^k_\gamma(i, j) \}. \]
		By \eqref{eq:DecompositionCond}, it holds that
		\begin{equation} 
				\label{eq:CmEst0}
				C^m_\beta(i, j) 
				\leq 
					\begin{cases} 
						C^k_\gamma(i, j) , & (i, j) \in \mathcal{Q} , \\ 
						C^n_\alpha(i, j) , & (i, j) \notin \mathcal{Q} . 
					\end{cases} 
			\end{equation}
Let $T \in A_\beta(\mathcal{A}, b_m)$ be arbitrary. Grothendieck's factorization theorem implies that there is $\widetilde{\beta} \in \mathfrak{C}^\infty$ with $\widetilde{\beta}_{\mid \ell(\beta)} = \beta$ such that  $T \in  A_{\widetilde{\beta}}(\mathcal{A}, b_m)$. 		
We claim that
			\begin{equation}
				\label{eq:IndexExtension1}
				\begin{gathered}
					\forall M > \ell(\gamma) ~ \exists N_{0, M} \in \N , R_{0, M} > 0 ~ \forall (i, j) \in \mathcal{Q} :  \\
					\widetilde{\beta}_3(M) \frac{a_{\beta_1, M, i}}{b_{m, \widetilde{\beta}_2(M), j}} \leq R_{0, M} \frac{a_{\gamma_1, M, i}}{b_{k, N_{0, M}, j}}, 
				\end{gathered}
			\end{equation}
		and
			\begin{equation}
				\label{eq:IndexExtension2}
				\begin{gathered}
					\forall M > \ell(\alpha) ~ \exists N_{1, M} \in \N , R_{1, M} > 0 ~ \forall (i, j) \notin \mathcal{Q} :  \\
					\widetilde{\beta}_3(M) \frac{a_{\beta_1, M, i}}{b_{m, \widetilde{\beta}_2(M), j}} \leq R_{1, M} \frac{a_{\alpha_1, M, i}}{b_{n, N_{1, M}, j}}.
				\end{gathered}
			\end{equation}
		Before proving these claims, let us show how they imply the result. We define $\widetilde{\gamma} \in \mathfrak{C}^\infty$ as $\widetilde{\gamma}_1 = \gamma_1$,
		
		\[ 
				\widetilde{\gamma}_2(M) = \begin{cases} \gamma_2(M) , & M \leq \ell(\gamma), \\ N_{0, M}, & M > \ell(\gamma), \end{cases} \quad
				\widetilde{\gamma}_3(M) = \begin{cases} \gamma_3(M) , & M \leq \ell(\gamma), \\ R_{0, M}, & M> \ell(\gamma), \end{cases}, \qquad M \in \N,
			\]
				and  $\widetilde{\alpha} \in \mathfrak{C}^\infty$ as $\widetilde{\alpha}_1 = \alpha_1$,
		
			\[
				\widetilde{\alpha}_2(M) = \begin{cases} \alpha_2(M) , & M \leq \ell(\alpha) , \\ N_{1, M} , & M >  \ell(\alpha), \end{cases} \quad
				\widetilde{\alpha}_3(M) = \begin{cases} \alpha_3(M) , & M \leq  \ell(\alpha) , \\ R_{1, M} , & M >  \ell(\alpha),\end{cases} \qquad M \in \N.
			\]			
Then, $\widetilde{\gamma}_{\mid \ell(\gamma)} = \gamma$ and $\widetilde{\alpha}_{\mid \ell(\alpha)} = \alpha$. By \eqref{eq:CmEst0}--\eqref{eq:IndexExtension2}, it holds that
			\begin{equation} 
				\label{eq:CmEst}
				C^m_{\widetilde{\beta}}(i, j) 
				\leq 
					\begin{cases} 
						C^k_{\widetilde{\gamma}}(i, j) , & (i, j) \in \mathcal{Q} , \\ 
						C^n_{\widetilde{\alpha}}(i, j) , & (i, j) \notin \mathcal{Q} . 
					\end{cases} 
			\end{equation}
	Let $T_{i, j} \in \C$, $(i, j) \in \mathcal{I} \times \mathcal{J}$, be as in Lemma  \ref{l:CharBoundedSp}(2) for $T \in  A_{\widetilde{\beta}}(\mathcal{A}, b_m)$.	We define
			\[ R_{i, j} = \begin{cases} T_{i, j} , & (i, j) \in \mathcal{Q} , \\ 0 , & (i, j) \notin \mathcal{Q} , \end{cases} \qquad S_{i, j} = \begin{cases} 0 , & (i, j) \in \mathcal{Q} , \\ T_{i, j} , & (i, j) \notin \mathcal{Q} , \end{cases} \qquad (i, j) \in \mathcal{I} \times \mathcal{J} . \]
Lemma  \ref{l:CharBoundedSp}(2) and \eqref{eq:CmEst} yield 
			\begin{equation} 
				\label{eq:last}
				|R_{i, j}| \leq C^k_{\widetilde{\gamma}}(i, j) , \qquad |S_{i, j}| \leq C^n_{\widetilde{\alpha}}(i, j) , \qquad \forall (i, j) \in \mathcal{I} \times \mathcal{J} . 
				\end{equation}
We define the linear mappings $R: \lambda^1(\mathcal{A}) \to k^\infty(b_k)$ and  $S: \lambda^1(\mathcal{A}) \to k^\infty(b_n)$  as 
$$
R(c) = \left(\sum_{i \in \mathcal{I}}  R_{i, j}c_i\right)_{j \in \mathcal{J}}, \quad S(c) = \left(\sum_{i \in \mathcal{I}} S_{i, j}c_i\right)_{j \in \mathcal{J}}, \qquad c = (c_i)_{i \in \mathcal{I}} \in \lambda^1(\mathcal{A}).
$$		
Then, $ T = R + S$ and, by Lemma  \ref{l:CharBoundedSp}(1) and  \eqref{eq:last},
$$
R \in A_{\widetilde{\gamma}}(\mathcal{A}, b_k) \subseteq  A_{\gamma}(\mathcal{A}, b_k), \qquad S \in A_{\widetilde{\alpha}}(\mathcal{A}, b_n) \subseteq  A_{\alpha}(\mathcal{A}, b_n).
$$ 	
		We now verify \eqref{eq:IndexExtension1} and \eqref{eq:IndexExtension2}. 
		We only show \eqref{eq:IndexExtension1} as \eqref{eq:IndexExtension2} can be shown similarly.
	If $\ell(\gamma) = \infty$, then  \eqref{eq:IndexExtension1} is trivially true. Hence, we assume that $\ell(\beta) \leq \ell(\gamma) < \infty$. Let $M > \ell(\gamma)$ be arbitrary. Since $y$ satisfies $(\DN)$, we may assume without loss of generality that
		\begin{equation}
\label{eq:yDN} 
			\forall m \in \N, \theta \in (0, 1) ~ \exists k \geq m, C> 0 ~ \forall j_1 \in \mathcal{J}_1 :  y_{m, j_1} \leq C y_{k, j_1}^{\theta} y_{0, j_1}^{1 - \theta}. 
		\end{equation}	
		By choosing $\theta = 1/2$, we find that
$$
			\forall m \in \N ~ \exists k \geq m, C> 0 ~ \forall j_1 \in \mathcal{J}_1 :  \frac{y_{m, j_1}}{y_{0,j_1}} \leq C \frac{y_{k, j_1}}{y_{m, j_1}}.
$$
From this, we obtain that there are $N_{0, M} \in \N$ and $C > 0$  such that 
			\[ \frac{y_{\widetilde{\beta}_2(M), j_1}}{y_{\beta_2(M'), j_1}} \leq C \frac{y_{N_{0, M}, j_1}}{y_{\gamma_2(M'), j_1}}, \qquad \forall M' \leq \ell(\beta) , \, j_1 \in \mathcal{J}_1 .  \]
Set	
			\[ R_{0, M} = C \widetilde{\beta}_3(M) \sup_{M' \leq \ell(\beta)} \frac{\gamma_3(M')}{\beta_3(M')} . \]
		
		Let $(i, j) \in \mathcal{Q}$ be arbitrary. Pick $M_{i, j} \leq \ell(\beta)$ such that 
			\[ C^m_\beta(i, j) = \beta_3(M_{i, j}) \frac{a_{\beta_1, M_{i, j}, i}}{b_{m, \beta_2(M_{i, j}), j}} . \]
		Then,
			\begin{align*}
				\frac{a_{\gamma_1, M, i}}{b_{k, N_{0, M}, j}}
				&=  \frac{a_{\gamma_1, M_{i, j}, i}}{b_{k, \gamma_2(M_{i, j}), j}} \frac{a_{\gamma_1, M, i}}{a_{\gamma_1, M_{i, j}, i}} \frac{b_{k, \gamma_2(M_{i, j}), j}}{b_{k, N_{0, M}, j}}  \\
				&\geq  \frac{C^m_\beta(i, j) }{\gamma_3(M_{i, j})}  \frac{w_{M_{i, j},i_1}}{w_{M, i_1}}\frac{y_{N_{0, M}, j_1}}{y_{\gamma_2(M_{i, j}), j_1}}  \\
				&\geq  \frac{1}{C} \frac{\beta_3(M_{i, j})}{\gamma_3(M_{i, j})}  \frac{a_{\beta_1, M_{i, j}, i}}{b_{m, \beta_2(M_{i, j}), j}} \frac{w_{M_{i, j},i_1}}{w_{M, i_1}}\frac{y_{\widetilde{\beta}_2(M), j_1}}{y_{\beta_2(M_{i, j}), j_1}}  \\
	&\geq \frac{\widetilde{\beta}_3(M)}{R_{0,M}}   \frac{v_{\beta_1, i_0}}{w_{M, i_1}} \frac{y_{\widetilde{\beta}_2(M), j_1}} {x_{m, j_0}}  			 =  \frac{\widetilde{\beta}_3(M)}{R_{0,M}}  \frac{a_{\beta_1, M, i}}{b_{m, \widetilde{\beta}_2(M), j}} .
			\end{align*}
		This shows \eqref{eq:IndexExtension1}.
	\end{proof}

	\begin{proposition}
		\label{p:Proj1L(X, Y)=0SeqSp}
Assume that $y$ satisfies $(\DN)$.		If
			\begin{equation}
				\label{eq:SuffCondProj^1=0}
				\begin{gathered}
					\exists \alpha \in \mathfrak{C}^{\infty} ~ \forall n \in \N ~ \exists m \in \N ~ \forall k \in \N, \beta \in \mathfrak{C}^n ~ \exists \gamma \in \mathfrak{C}^n ~ \forall (i, j) \in \mathcal{I} \times \mathcal{J} : \\
					C^m_\beta(i, j) \leq \min \{ C^k_\gamma(i, j) , C^n_{\alpha_{\mid n}}( i, j) \} ,
				\end{gathered}
			\end{equation}
			then $\Proj^{1} L(\lambda^1(\mathcal{A}), \lambda^\infty(\mathcal{B})) = 0$.
	\end{proposition}
	
	\begin{proof}
		By Proposition \ref{p:Proj1L(X, Y)=0}, it  suffices to show that (the spectral mappings, being inclusions, are omitted)
		\[
				\begin{gathered}
					\exists \alpha \in \mathfrak{C}^{\infty} ~ \forall n \in \N ~ \exists m \in \N ~ \forall k \in \N, \beta \in \mathfrak{C}^n ~ \exists \gamma \in \mathfrak{C}^n : \\
					A_\beta(\mathcal{A},b_m) \subseteq A_\gamma(\mathcal{A},b_k) + A_{\alpha_{\mid n}}(\mathcal{A},b_n). 
				\end{gathered}
			\]
This follows from Lemma \ref{l:CharOfDecomposition} and \eqref{eq:SuffCondProj^1=0}.
	\end{proof}
	
We are ready to show Theorem \ref{t:SplittingResult}.
	
	\begin{proof}[Proof of Theorem \ref{t:SplittingResult}]
		In view of Proposition \ref{p:Proj^1=0=>Ext^1=0}, Lemma \ref{l:SeqSpLocallySplitting}, and Proposition \ref{p:Proj1L(X, Y)=0SeqSp}, it suffices to show   \eqref{eq:SuffCondProj^1=0}.
	
		As $v$ and $y$ satisfy $(\DN)$, we may assume without loss of generality that
		\begin{eqnarray}
				\label{eq:vDN} 
				\forall m \in \N, \theta \in (0, 1) ~ \exists k \geq m, C> 0 ~ \forall i_0 \in \mathcal{I}_0 :  v_{m, i_0} \leq C v_{k, i_0}^{\theta} v_{0, i_0}^{1 - \theta}, \\ \label{eq:yDN}
				\forall m \in \N, \theta \in (0, 1) ~ \exists k \geq m, C> 0 ~ \forall j_1 \in \mathcal{J}_1 :  y_{m, j_1} \leq C y_{k, j_1}^{\theta} y_{0, j_1}^{1 - \theta}. 
			\end{eqnarray}
We define $\alpha \in \mathfrak{C}^\infty$ as
			\[ \alpha_1 = 0, \qquad \alpha_2(M) = 0 , \qquad \alpha_3(M) = 1 , \qquad M \in \N . \]
		Let $n \in \N$ be arbitrary. 
		As $x$ satisfies $(\Omega)$, there is  $m \geq n$ such that
			\begin{equation}
				\label{eq:xOmega}
				\forall k \geq m ~ \exists \theta \in (0, 1), C > 0 ~ \forall j_0 \in \mathcal{J}_0 :  x_{k, j_0}^{\theta} x_{n, j_0}^{1 - \theta} \leq C x_{m, j_0} .
			\end{equation}
		Let $k \in \N$ and $\beta \in \mathfrak{C}^n$ be arbitrary. 
		By \eqref{eq:xOmega}, there are $\theta \in (0, 1)$ and  $C_0 > 0$ such that
		\begin{equation}
\label{Omega2}
		\left ( \frac{x_{n,j_0}}{x_{m,j_0}}\right)^{\frac{1-\theta}{\theta}} \leq C_0 \frac{x_{m,j_0}}{x_{k,j_0}}, \qquad \forall j_0 \in \mathcal{J}_0.
		\end{equation}
	Set  $p = \max_{M \leq n}\beta_2(M)$. By \eqref{eq:vDN} with $m = \beta_1$ and \eqref{eq:yDN} with $m =p$, there are $k_1,k_2 \in \N$ and  $C_1,C_2 > 0$ such that
		\begin{eqnarray}
		\label{DN2}
		&& \frac{v_{\beta_1,i_0}}{v_{k_1,i_0}} \leq C_1 \left (\frac{v_{0,i_0}}{v_{\beta_1,i_0}}\right)^{\frac{1-\theta}{\theta}}
, \quad  \forall i_0 \in \mathcal{I}_0, \\ \label{DN3}
	&& y^{1/\theta}_{p,j_1} \leq C_2 y_{k_2,j_1} y^{\frac{1-\theta}{\theta}}_{0,j_1}, \qquad  \forall j_1 \in \mathcal{J}_1.
\end{eqnarray}
Set $C' = C_0C_1C_2$. We define $\gamma \in \mathfrak{C}^n$ as
	\[ \gamma_1 = k_1 , \qquad \gamma_2(M) = k_2 , \qquad \gamma_3(M) = C', \qquad M \leq n. \]
Set
		$$
		\mathcal{Q} = \left\{ (i, j) \in \mathcal{I} \times \mathcal{J} \mid  C^m_\beta(i,j) > C^n_{\alpha_{\mid n}}(i, j)\right \}.
		$$
It suffices to show that
\begin{equation}
\label{TPw}
		C^m_\beta(i,j) \leq C^k_\gamma(i,j), \qquad \forall (i, j) \in \mathcal{Q}.
\end{equation}
 For all $(i, j) \in \mathcal{I} \times \mathcal{J}$ it holds that
\[ C^n_{\alpha_{\mid n}}(i, j) = \frac{v_{0, i_0}}{x_{n, j_0}} \frac{y_{0, j_1}}{w_{n, i_1}}, \quad C^m_\beta(i,j)  = \frac{v_{\beta_1, i_0}}{x_{m, j_0}} \min_{M \leq n} \frac{y_{\beta_2(M), j_1}}{w_{M, i_1}}, \quad C^k_{\gamma}(i, j) = C'\frac{v_{k_1,i_0}}{x_{k, j_0}} \frac{y_{k_2, j_1}}{w_{n, i_1}}.  \]		
Note that \eqref{TPw} is equivalent to	
$$
\frac{v_{\beta_1, i_0}}{v_{k_1, i_0}}\min_{M \leq n} \frac{y_{\beta_2(M), j_1}}{w_{M, i_1}} \leq  C'\frac{x_{m, j_0}}{x_{k, j_0}}  \frac{y_{k_2, j_1}}{w_{n, i_1}}, \qquad  \forall (i, j) \in \mathcal{Q}.
$$
We now show the latter inequality.	
From the definition of $\mathcal{Q}$, we find that
$$
\frac{v_{0, i_0}}{v_{\beta_1, i_0}} <  \frac{x_{n, j_0}}{x_{m, j_0}} \frac{w_{n, i_1}}{y_{0, j_1}}\min_{M \leq n} \frac{y_{\beta_2(M), j_1}}{w_{M, i_1}}, \qquad  \forall (i, j) \in \mathcal{Q}.
$$			
Combining this with \eqref{Omega2}--\eqref{DN3} ,we obtain that for all $ (i, j) \in \mathcal{Q}$
\begin{eqnarray*}
\frac{v_{\beta_1, i_0}}{v_{k_1, i_0}}\min_{M \leq n} \frac{y_{\beta_2(M), j_1}}{w_{M, i_1}} &\leq&  C_1 \left (\frac{v_{0,i_0}}{v_{\beta_1,i_0}}\right)^{\frac{1-\theta}{\theta}}\min_{M \leq n} \frac{y_{\beta_2(M), j_1}}{w_{M, i_1}} \\
&<& C_1 \left ( \frac{x_{n, j_0}}{x_{m, j_0}} \right)^{\frac{1-\theta}{\theta}} \left ( \frac{w_{n, i_1}}{y_{0, j_1}}\ \right)^{\frac{1-\theta}{\theta}}\min_{M \leq n} \left ( \frac{y_{\beta_2(M), j_1}}{w_{M, i_1}}\right)^{\frac{1}{\theta}} \\
&\leq& C'\frac{x_{m, j_0}}{x_{k, j_0}}  \frac{y_{k_2, j_1}}{w_{n, i_1}}.
\end{eqnarray*}
\end{proof}

\section{A  Pe\l czy\'nski-Vogt decomposition result for $(\PLS)$-type power series spaces of infinite type}\label{sect-PV}

The goal of this section is to establish our first main result: a Pełczyński-Vogt decomposition theorem for the spaces $\Lambda_\infty(\alpha,\beta)$. Our approach adapts the method used in the proof of the decomposition result for power series spaces of infinite type in \cite[Satz 1.3]{V-IsomorphSatzPotRaum}, which relies on the splitting theory for such spaces (see Proposition \ref{p:SplittingWeightMatrices}(1)). We begin with the following key lemma.

	\begin{lemma}
		 \label{l:ExLambda=Lambda}
		Let $\Lambda$ be a $(\PLS)$-space satisfying the following conditions:
			\begin{itemize}
				\item[(C1)]  $\Lambda \times \Lambda \cong \Lambda$.
				\item[(C2)] There exists an exact sequence
					\[ \SES{\Lambda}{\Lambda}{\Lambda^{\N}}{}{}. \]
				\item[(C3)] $\ExtPLS^{1}(\Lambda, \Lambda) = 0$.
			\end{itemize}
		Let $X$ be a lcHs. If $X \compl \Lambda$, then $X \times  \Lambda \cong \Lambda$. 
	\end{lemma}
	\begin{proof}
Since  $X \compl \Lambda$, $X$ is a $(\PLS)$-space and there is a $(\PLS)$-space $Y$ such that  $\Lambda \cong X \times Y$. Then,
		\begin{equation}
				\label{eq:ExLambda^N=Lambda^N}  \Lambda^{\N} \cong X^{\N} \times Y^{\N} \cong (X \times X^{\N}) \times Y^{\N} \cong X \times (X \times Y)^{\N}  \cong X \times \Lambda^{\N}. 
				\end{equation}
By (C2), there exists an exact sequence
			\begin{equation}
				\label{eq:C2} 
				\SES{\Lambda}{\Lambda}{\Lambda^{\N}}{\iota}{q}. 
			\end{equation}
		From here, we obtain the  exact sequence
			\[ \SES{\Lambda}{X \times \Lambda}{X \times \Lambda^{\N}}{\tilde{\iota}}{\tilde{q}}, \]
		where $\tilde{\iota}(\lambda) = (0, \lambda)$ and $\tilde{q}((x, \lambda)) = (x, q(\lambda))$ for $x \in X$, $\lambda \in \Lambda$. 
By combining this with \eqref{eq:ExLambda^N=Lambda^N},  we find an exact sequence of the form
			\begin{equation}
				\label{eq:SES1} 
				\SES{\Lambda}{X \times \Lambda}{\Lambda^{\N}}{\iota}{T}. 
			\end{equation}
			Define the $(\PLS)$-space
			$$
			Z = \{ (\lambda,y) \in \Lambda \times (X \times \Lambda) \mid q(\lambda) = T(y) \}.
			$$
We obtain the following commutative diagram of  exact sequences
			\[
				\begin{tikzcd}
					& & 0 & 0 & \\
					0 \arrow{r} & \Lambda \arrow{r} & \Lambda \arrow{r}{q} \arrow{u} & \Lambda^{\N} \arrow{r} \arrow{u} & 0 \\
					0 \arrow{r} & \Lambda \arrow{r}  & Z \arrow{r}{p_{2}} \arrow{u}{p_1} & X \times \Lambda \arrow{r} \arrow{u}{T} & 0 \\
					& & \Lambda \arrow{u} & \Lambda \arrow{u} & \\
					& & 0 \arrow{u} & 0 \arrow{u} & 
				\end{tikzcd}
			\]
		where $p_1(\lambda,y) = \lambda$ and $p_2(\lambda,y) = y$ for $\lambda \in \Lambda, y \in X \times \Lambda$. Conditions (C1) and (C3) give
			\[ Z \cong \Lambda \times \Lambda \cong \Lambda . \]
By (C1), it holds that $X \times \Lambda \compl \Lambda \times \Lambda \cong \Lambda$. Hence, from (C3) we obtain that also $\ExtPLS^1 (X \times \Lambda, \Lambda) = 0$. 
		Therefore, another application of (C1) yields
			\[ Z \cong (X \times \Lambda) \times \Lambda \cong X \times \Lambda .  \]
		We conclude $X \times \Lambda \cong Z \cong \Lambda$.
	\end{proof}

	\begin{proposition}\label{PVabstract}
		Let $\Lambda$ be a $(\PLS)$-space satisfying \emph{(C1)},  \emph{(C2)}, and  \emph{(C3)}. Let $X$ be a lcHs.	 
		If $X \compl \Lambda$ and $ \Lambda \compl X$, then $X \cong \Lambda$.
	\end{proposition}

	\begin{proof}
	Since  $X \compl \Lambda$, $X$ is a $(\PLS)$-space. As $\Lambda \compl X$, there is a $(\PLS)$-space $Y$ such that $X \cong Y \times \Lambda$. Then, $Y \compl X$ and, as $X \compl \Lambda$, also $Y \compl \Lambda$.
Lemma \ref{l:ExLambda=Lambda} implies that $X \cong Y \times \Lambda \cong \Lambda$.
	\end{proof}

We are ready to show our first main result.
	\begin{theorem}
		\label{t:PelczynskiPLS}
		Let $\alpha, \beta$ be exponent sequences satisfying \emph{(N)} with $\alpha$ stable.
		Let $X$ be a lcHs.	
		If $X \compl \Lambda_{\infty}(\alpha, \beta)$ and $\Lambda_{\infty}(\alpha, \beta) \compl X$, then $X \cong \Lambda_{\infty}(\alpha, \beta)$.
	\end{theorem}
	
	\begin{proof}[Proof of Theorem \ref{t:PelczynskiPLS}]
In view of Proposition \ref{PVabstract}, it suffices to show that the $(\PLS)$-space $\Lambda_{\infty}(\alpha, \beta)$ satisfies the conditions (C1), (C2), and (C3). \\
(C1): As $\alpha$ is stable, it holds that $\Lambda_{\infty}(\alpha) \cong \Lambda_{\infty}(\alpha) \times \Lambda_{\infty}(\alpha)$. By Lemma \ref{lemma:tensor}, it holds that
			\begin{align*}
			 \Lambda_{\infty}(\alpha, \beta) &\cong \Lambda_{\infty}(\alpha) \varepsilon  \Lambda'_{\infty}(\beta) \cong (\Lambda_{\infty}(\alpha) \times \Lambda_{\infty}(\alpha)) \varepsilon  \Lambda'_{\infty}(\beta) \\
			 &\cong (\Lambda_{\infty}(\alpha) \varepsilon  \Lambda'_{\infty}(\beta) ) \times (\Lambda_{\infty}(\alpha) \varepsilon  \Lambda'_{\infty}(\beta))  = \Lambda_{\infty}(\alpha, \beta)  \times \Lambda_{\infty}(\alpha, \beta) 
			\end{align*}
(C2): As $\alpha$ is stable, \cite[Satz 2.4]{V-W-CharUnterraumQuotientraumNuklStabPotRaumInfty}  implies that there exists an exact  sequence
$$
\SES{\Lambda_\infty(\alpha)}{\Lambda_\infty(\alpha)}{\Lambda_\infty(\alpha)^{\N}}{}{}.
$$
Since $\Ext^1(\Lambda_\infty(\beta), \Lambda_\infty(\alpha)) = 0$ (Proposition  \ref{p:SplittingWeightMatrices}(1)), we obtain from Lemma \ref{charEXT} (see also \cite[Theorem 6.3]{D-N-ExtResLBSpSurjTensMap}) that the sequence 
$$
\SES{\Lambda_\infty(\alpha)\varepsilon \Lambda'_\infty(\beta) }{\Lambda_\infty(\alpha) \varepsilon \Lambda'_\infty(\beta)}{\Lambda_\infty(\alpha)^{\N}  \varepsilon \Lambda'_\infty(\beta)}{}{}
$$
is algebraically exact.  By using Lemma \ref{lemma:tensor} and $\Lambda_\infty(\alpha)^{\N} \varepsilon \Lambda'_\infty(\beta) \cong (\Lambda_\infty(\alpha) \varepsilon \Lambda'_\infty(\beta))^{\N}$, we find an algebraically exact sequence
$$
\SES{\Lambda_\infty(\alpha,\beta)}{\Lambda_\infty(\alpha,\beta)}{\Lambda_\infty(\alpha, \beta)^{\N}}{}{}.
$$
We now show that this sequence is also topologically exact. We first show that an algebraically exact sequence of ultrabornological $(\PLS)$-spaces is automatically topologically exact. Since every $(\PLS)$-space has a strict ordered web, De Wilde's open mapping theorem \cite[Theorem 24.30]{M-V-IntroFunctAnal} implies that any surjective continuous linear mapping $X \to Y$, with $X$ and $Y$ ($\PLS$)-spaces and $Y$ ultrabornological,  is a topological homomorphism. The conclusion now follows from the fact that a closed subspace $A$
of an ultrabornological $(\PLS)$-space $X$ is ultrabornological if and only if $X/A$ is complete \cite[Corollary 1.4]{D-V-SplitThSpDist}. Hence, it suffices to show that $\Lambda_\infty(\alpha,\beta)$ is ultrabornological (note that the countable product of ultrabornological spaces is again ultrabornological). As $\Ext^1(\Lambda_\infty(\beta), \Lambda_\infty(\alpha)) = 0$ (Proposition  \ref{p:SplittingWeightMatrices}(1)), this follows from Lemma \ref{lemma:tensor} and  \cite[Theorem 3.3.4 and Proposition 5.1.5]{W-DerivFuncFunctAnal}. \\
(C3): This follows from Corollary \ref{c:SplittingResultPS}.
\end{proof}

We end this section with a remark on the analogue of Theorem \ref{t:PelczynskiPLS} in the finite type case. Consider the following condition on an exponent sequence $\alpha$:
\begin{equation}
\lim_{i \to \infty} \frac{\log(1+i)}{\alpha_i} = 0.
\label{nucls}
\end{equation}
This condition is equivalent to $\Lambda^p_0(\alpha)$, $p \in \{1, \infty\}$, being nuclear \cite[Proposition 29.6(2)]{M-V-IntroFunctAnal}. If  $\alpha$ and $\beta$ are exponent  sequences satisfying \eqref{nucls}, then  $\Lambda^\infty_0(\alpha) =\Lambda^1_0(\alpha)$ and  $\Lambda^\infty_0(\alpha, \beta) =\Lambda^1_0(\alpha, \beta)$, and we simply write $\Lambda_0(\alpha)$ and   $\Lambda_0(\alpha,\beta)$, respectively, for these spaces.
\begin{remark}\label{openqabst}
As  mentioned in the introduction, the following decomposition result for power series of finite type is known \cite[Satz 1.4]{V-IsomorphSatzPotRaum}:  
\emph{Let $\alpha$ be an exponent  sequence satisfying \eqref{nucls} and let $X$ be a lcHs.	If $X \compl \Lambda_{0}(\alpha)$ and $\Lambda_{0}(\alpha) \compl X$, then $X \cong \Lambda_{0}(\alpha)$}.

It seems natural to expect that the analogue of Theorem \ref{t:PelczynskiPLS} also holds in the finite type case,  but we do not know how to prove this. More precisely, there is the following open problem: 
\emph{Let $\alpha$  and $\beta$  be exponent sequences satisfying \eqref{nucls} and let $X$ be a lcHs. If  $X \compl \Lambda_{0}(\alpha,\beta)$ and $\Lambda_{0}(\alpha,\beta) \compl X$, is it true that $X \cong \Lambda_{0}(\alpha, \beta)$? } 
 
 Note that, under the assumption $\sup_{i \in \N} \alpha_{i+1}/\alpha_i < \infty$,  this does not follow from Proposition \ref{PVabstract} because  $\ExtPLS^{1}(\Lambda_{0}(\alpha, \beta), \Lambda_{0}(\alpha, \beta)) \neq  0$ (otherwise,  $\Lambda_{0}(\alpha) \compl \Lambda_{0}(\alpha, \beta)$ would imply that   $\Ext^{1}(\Lambda_{0}(\alpha), \Lambda_{0}(\alpha))
 = 0$, which is known to be false \cite[Corollary 4.4]{V-Ext1Frechet}). It would be interesting to investigate whether the method used to prove the above decomposition result for power series of finite type can be extended to the setting of $(\PLS)$-spaces.
 \end{remark}

\section{Sequence space representations of multiplier spaces of Gelfand-Shilov spaces of Beurling type}\label{sect-appl}
In this section, we obtain sequence space representations of multiplier spaces of Gelfand-Shilov spaces of Beurling type. Our proof is based on the  Pe\l czy\'nski-Vogt decomposition result for the spaces $\Lambda_\infty(\alpha,\beta)$ (Theorem \ref{t:PelczynskiPLS}) and properties of Gabor frames \cite{G-FoundTFAnalysis}.

\subsection{Gelfand-Shilov spaces}

	A \emph{weight function} $\omega$  is a non-decreasing continuous function $\omega : [0, \infty) \to [0, \infty)$ satisfying the following properties:
		\begin{itemize}
			\item[$(\alpha)$] $\omega(2 t) = O(\omega(t))$ as $t \to \infty$;
			\item[$(\gamma)$] $\log t = o(\omega(t))$ as $t \to \infty$.
		\end{itemize}
We say that  $\omega$ is a \emph{BMT-weight function} \cite{B-M-T-UltradiffFuncFourierAnal} if it is a weight function that, in addition,  satisfies $\omega_{\mid [0,1]} = 0$ and
		\begin{itemize}
			\item[$(\delta)$] $\phi : [0, \infty) \to [0, \infty)$, $\phi(x) = \omega(e^x)$, is convex.
		\end{itemize}

Let $\omega$ be a BMT-weight function.	
The \emph{Young conjugate} of $\phi$ is defined as
		\[\phi^* : [0, \infty) \to [0, \infty), \quad \phi^*(y) = \sup_{x \geq 0} \{x y - \phi(x)\}. \]
	The function $\phi^*$ is convex and increasing, $(\phi^*)^* = \phi$, and the function $y \mapsto \phi^*(y) / y$ is increasing on $[0, \infty)$ and tends to infinity as $y \to \infty$.
	\begin{example}\label{GevreyW}
The \emph{Gevrey weight of order $\mu$}, $\mu >0$, is defined as 
$$
\omega_\mu(t)  = \max \{ 0, t^{1/\mu} -1\}, \qquad t \geq 0.
$$
\end{example}

Let $\omega$  be a BMT-weight function and let $\eta$ be a weight function. For $h >0$ and $\lambda \in \R$ we  write $\mathcal{S}_{\eta,\lambda}^{\omega,h}$ for the Banach space consisting of all $\varphi \in C^\infty(\R)$ such that
$$
\sup_{p \in \N} \sup_{x \in \R} |\varphi^{(p)}(x)| \exp \left( \frac{1}{\lambda} \eta(x) -  \frac{1}{h} \phi^{*}(h p) ) \right) < \infty.
$$
We define the \emph{Gelfand-Shilov spaces (of Beurling and Roumieu type)} as
\[
\mathcal{S}_{(\eta)}^{(\omega)}=\varprojlim_{h \to0^+} \mathcal{S}_{\eta,h}^{\omega,h} \qquad \mbox{and}
\qquad \mathcal{S}_{\{\eta\}}^{\{\omega\}}=\varinjlim_{h\to +\infty} \mathcal{S}_{\eta,h}^{\omega,h}.
\]
We use $\mathcal{S}_{[\eta]}^{[\omega]}$ as a common notation for $\mathcal{S}_{(\eta)}^{(\omega)}$ and $\mathcal{S}_{\{\eta\}}^{\{\omega\}}$; a similar convention is used for other spaces as well. 
 \begin{example}\label{CGS}
Let $\mu, \tau >0$. For $h>0$ and $\lambda \in \R$ we define  $\mathcal{S}_{\tau,\lambda}^{\mu,h}$ as the Banach space consisting of all $\varphi \in C^\infty(\R)$ such that
$$
 \sup_{p \in \N} \sup_{x \in \R} \frac{|\varphi^{(p)}(x)|e^{\frac{1}{\lambda}|x|^{1/\tau}}}{h^{p}p!^\mu} < \infty.
$$
We set 
\[
\Sigma^{\mu}_{\tau}=\varprojlim_{h \to0^+} \mathcal{S}_{\tau,h}^{\mu,h} \qquad \mbox{and}
\qquad \mathcal{S}^{\mu}_{\tau}=\varinjlim_{h\to +\infty} \mathcal{S}_{\tau,h}^{\mu,h}.
\]
We already considered the space   $\Sigma^{\mu}_{\tau}$ in the introduction.  The spaces  $\mathcal{S}^{\mu}_\tau$ were  introduced by Gelfand and Shilov \cite{G-S-GenFunc2}. Using the same notation as in Example \ref{GevreyW}, it holds that $\Sigma^\mu_\tau = \mathcal{S}^{(\omega_\mu)}_{(\omega_\tau)}$ and $\mathcal{S}^\mu_\tau = \mathcal{S}^{\{\omega_\mu\}}_{\{\omega_\tau\}}$.
\end{example}

\begin{remark}
It is an open problem to characterize the non-triviality of the spaces $\mathcal{S}^{[\omega]}_{[\eta]}$  in terms of $\omega$ and $\eta$. It is known that $\Sigma^{\mu}_{\tau}$ is non-trivial if and only if $ \mu + \tau >1$, while $\mathcal{S}^{\mu}_{\tau}$ is non-trivial if and only $\mu + \tau \geq 1$ (cf.\  \cite[Section 8]{G-S-GenFunc2}).
\end{remark}
\subsection{Multiplier spaces of $\mathcal{S}^{[\omega]}_{[\eta]}$} 
Let $\omega$  be a BMT-weight function and let $\eta$ be a weight function. We define 
\[
\mathcal{Z}_{(\eta)}^{(\omega)}=\varprojlim_{h \to0^+} \varinjlim_{\lambda \to0^-}  \mathcal{S}_{\eta,\lambda}^{\omega,h} \qquad \mbox{and}
\qquad \mathcal{Z}_{\{\eta\}}^{\{\omega\}}=\varprojlim_{\lambda\to -\infty}\varinjlim_{h\to +\infty} \mathcal{S}_{\eta,\lambda}^{\omega,h}.
\]
We refer to our previous articles \cite{D-N-WeighPLBUltradiffFuncMultSp, D-N-BarrelWeighPLBUltradiffFunc} for a detailed study of the linear topological and structural properties of the spaces $\mathcal{Z}_{(\eta)}^{(\omega)}$ and  $\mathcal{Z}_{\{\eta\}}^{\{\omega\}}$. 

\begin{example}\label{CGS2} 
Let $\mu, \tau >0$. 
We set 
\[
\mathcal{Z}_{(\tau)}^{(\mu)}=\varprojlim_{h \to0^+} \varinjlim_{\lambda \to 0^-} \mathcal{S}_{\tau,h}^{\mu,\lambda} \qquad \mbox{and}
\qquad \mathcal{Z}_{\{\tau\}}^{\{\mu\}}=\varprojlim_{\lambda \to -\infty} \varinjlim_{h \to +\infty} \mathcal{S}_{\tau,h}^{\mu,\lambda}.
\]
We already considered the space  $\mathcal{Z}_{(\tau)}^{(\mu)}$ in the introduction.  Using the same notation as in Example \ref{GevreyW}, it holds that $\mathcal{Z}_{(\tau)}^{(\mu)}= \mathcal{Z}^{(\omega_\mu)}_{(\omega_\tau)}$ and $\mathcal{Z}_{\{\tau\}}^{\{\mu\}} = \mathcal{Z}^{\{\omega_\mu\}}_{\{\omega_\tau\}}$.
\end{example}

The main motivation to consider the spaces $\mathcal{Z}_{(\eta)}^{(\omega)}$ and  $\mathcal{Z}_{\{\eta\}}^{\{\omega\}}$ stems from the fact that, as shown in \cite{D-N-WeighPLBUltradiffFuncMultSp}, they are the multiplier spaces of the Gelfand-Shilov spaces $\mathcal{S}_{(\eta)}^{(\omega)}$ and  $\mathcal{S}_{\{\eta\}}^{\{\omega\}}$, respectively. We now explain this in more detail.

 We write $\mathcal{S}^{\prime [\omega]}_{[\eta]}$ for the dual of $\mathcal{S}^{[\omega]}_{[\eta]}$. We may consider $\mathcal{Z}^{[\omega]}_{[\eta]}$ as a vector subspace of $\mathcal{S}^{\prime [\omega]}_{[\eta]}$ \cite[Lemma 5.2]{D-N-WeighPLBUltradiffFuncMultSp}. 
 The multiplication mapping 
 	\[  \mathcal{S}_{[\eta]}^{[\omega]} \times \mathcal{S}_{[\eta]}^{[\omega]} \to \mathcal{S}_{[\eta]}^{[\omega]}, \qquad (\varphi,\psi) \mapsto \varphi \cdot \psi , \] 
is a well-defined, separately continuous bilinear mapping. Hence, given $\varphi \in \mathcal{S}^{[\omega]}_{[\eta]}$ and $f \in \mathcal{S}^{\prime [\omega]}_{[\eta]}$, we can define $\varphi \cdot f \in \mathcal{S}^{\prime [\omega]}_{[\eta]}$ via duality, i.e., $\langle \varphi \cdot f, \psi \rangle = \langle f, \varphi \cdot \psi \rangle$ for all $\psi \in \mathcal{S}^{[\omega]}_{[\eta]}$.  Moreover, the mapping
\begin{equation}
\label{contprod}
\mathcal{S}^{[\omega]}_{[\eta]} \to \mathcal{S}^{ \prime [\omega]}_{[\eta]}, \quad \varphi \mapsto \varphi \cdot f
\end{equation}
is continuous. We define the \emph{multiplier space of $\mathcal{S}^{[\omega]}_{[\eta]}$} as
		\[ \mathcal{O}_M(\mathcal{S}^{[\omega]}_{[\eta]}) = \left\{ f \in \mathcal{S}^{\prime [\omega]}_{[\eta]} \mid \varphi \cdot f \in \mathcal{S}^{[\omega]}_{[\eta]} \text{ for all } \varphi \in \mathcal{S}^{[\omega]}_{[\eta]} \right\}. \]
De Wilde's closed graph theorem and the continuity of the mapping \eqref{contprod} imply that, for all $f \in  \mathcal{O}_M(\mathcal{S}^{[\omega]}_{[\eta]})$, the mapping
$$
\mathcal{S}^{[\omega]}_{[\eta]} \to \mathcal{S}^{[\omega]}_{[\eta]}, \quad \varphi \mapsto \varphi \cdot f,
$$
is continuous. We endow $\mathcal{O}_M(\mathcal{S}^{[\omega]}_{[\eta]})$ with the topology induced by the embedding
$$
\mathcal{O}_M(\mathcal{S}^{[\omega]}_{[\eta]}) \to L(\mathcal{S}^{[\omega]}_{[\eta]}, \mathcal{S}^{[\omega]}_{[\eta]}), \quad f \mapsto (\varphi \mapsto \varphi \cdot f). 
$$
If $\mathcal{S}^{[\omega]}_{[\eta]} \neq \{0\}$, then $\mathcal{Z}^{[\omega]}_{[\eta]} =  \mathcal{O}_M(\mathcal{S}^{[\omega]}_{[\eta]})$ as lcHs \cite[Theorem 5.4]{D-N-WeighPLBUltradiffFuncMultSp}.
\subsection{Gabor frames} The translation and modulation operators are denoted by $T_{x} f(t) = f(t - x)$ and $M_{\xi} f(t) = e^{2 \pi i \xi t} f(t)$, respectively, for $x, \xi \in \R$. Let  $\psi \in L^{2}(\R)$ and $a, b > 0$. The set of time-frequency shifts
	\[ \mathcal{G}(\psi, a, b) = \{ T_{ak} M_{bl} \psi \mid (k,l) \in \Z^2 \} \]
is called a \emph{Gabor frame} for $L^{2}(\R)$ if there exist $A, B > 0$ such that
	\[ A \|f\|^{2}_{L^{2}} \leq \sum_{ (k,l) \in \Z^2} \left| (f, T_{ak}M_{bl} \psi)_{L^2}\right|^{2} \leq B \|f \|^{2}_{L^{2}} , \qquad \forall f  \in L^{2}(\R). \]
Let $\psi \in \mathcal{S}$. The \emph{analysis operator}
$$
C_\psi = C^{a,b}_\psi:  L^2(\R) \rightarrow \ell^2(\Z^2), ~  f \mapsto ((f, T_{ak}M_{bl} \psi)_{L^2})_{(k,l) \in \Z^2},
$$
and the \emph{synthesis operator}
$$
D_\psi = D^{a,b}_\psi: \ell^2(\Z^2) \rightarrow L^2(\R), ~  (c_{k,l})_{ (k,l) \in \Z^2} \mapsto \sum_{ (k,l) \in \Z^2} c_{k,l} T_{ak} M_{bl} \psi
$$
are continuous. 
Let $\gamma \in \mathcal{S}$. We call  \emph{$(\psi,\gamma)$ a pair of dual windows  (on $a\Z \times b\Z$)} if
\begin{equation}
D_\gamma \circ C_\psi = \operatorname{id}_{L^2(\R)}.
		\label{comp11}
\end{equation}
In such a case, $(\gamma, \psi)$ is also a pair of dual windows and both $\mathcal{G}(\psi, a, b)$ and $\mathcal{G}(\gamma, a, b)$ are Gabor frames. 

We will use the following characterization of pairs of dual windows,  known as the \emph{Wexler-Raz biorthogonality relations}:
\begin{theorem}\cite[Theorem 7.3.1 and the subsequent remark]{G-FoundTFAnalysis}
		\label{l:WZBiOrthRel}
		Let $\psi, \gamma \in  \mathcal{S}$ and let $a,b > 0$. Then,  $(\psi,\gamma)$  is a pair of dual windows on $a\Z \times b\Z$ if and only if
		$$
		\frac{1}{ab}(  T_{\frac{k}{b}}M_{\frac{l}{a}}\psi, T_{\frac{k'}{b}} M_{\frac{l'}{a}} \gamma)_{L^{2}} = \delta_{k,k'} \delta_{l,l'}, \qquad \forall (k, l), (k,' l') \in \Z^2,
$$
or, equivalently,
\begin{equation}
\frac{1}{ab}C^{\frac{1}{b}, \frac{1}{a}}_\psi \circ D^{\frac{1}{b},\frac{1}{a}}_\gamma =  \operatorname{id}_{\ell^2\left( \Z^2 \right)}.
		\label{WR}
\end{equation}
\end{theorem}		

Let $\omega$  be a BMT-weight function and let $\eta$ be a weight function.  The space  $\mathcal{S}^{[\omega]}_{[\eta]}$ is called \emph{Gabor accessible} if there exist $\psi, \gamma \in \mathcal{S}^{[\omega]}_{[\eta]}$ and $a, b > 0$ such that $(\psi,\gamma)$  is a pair of dual windows on $a\Z \times b\Z$. We now use a fundamental result of Janssen \cite{Janssen} (see also \cite{B-J-GaborUnimodWindDecay})  to give a  growth condition on  $\omega$ and $\eta$ which ensures that  $\mathcal{S}^{[\omega]}_{[\eta]}$ is Gabor accessible.

\begin{proposition}\label{GA-1} Let $\omega$  be a BMT-weight function and $\eta$ be a weight function. The space $\mathcal{S}^{[\omega]}_{[\eta]}$ is Gabor accessible if $\omega(t) = o(t^2)$ and $\eta(t) = o(t^2)$ ($\omega(t) = O(t^2)$ and $\eta(t) = O(t^2)$).
\end{proposition}
\begin{proof}
Since $\mathcal{S}^{1/2}_{1/2} \subseteq \mathcal{S}^{[\omega]}_{[\eta]}$, it suffices to show that $\mathcal{S}^{1/2}_{1/2}$ is Gabor accessible. Let $\psi(x) = e^{-\pi x^2}$, $x \in \R$, be the Gaussian and fix $a,b > 0$ with $ab <1$. Then,  $\psi \in  \mathcal{S}^{1/2}_{1/2}$ and, by \cite[Proposition B and its proof]{Janssen},  there exists $\gamma \in  \mathcal{S}^{1/2}_{1/2}$ such that $(\psi, \gamma)$ is a pair of dual windows on $a\Z \times b\Z$.  
\end{proof}

\begin{remark}\label{GAR}
It is an open problem whether  $\mathcal{S}^{[\omega]}_{[\eta]}$ is Gabor accessible. In particular, it is not known whether $\mathcal{S}^\mu_\tau$, $\mu + \tau \geq 1$,  is always Gabor accessible.
\end{remark}
Finally, we recall a result from \cite{D-N-WeighPLBUltradiffFuncMultSp} about the mapping properties of the analysis and synthesis operators on $\mathcal{Z}^{[\omega]}_{[\eta]}$. To this end, we define, for $a,b >0$, the weight matrices
$$
	 \mathcal{A}^{a,b}_{(\omega), (\eta)} = ((e^{n \omega(b |l|) - N \eta(a |k|)})_{k, l \in \Z})_{n, N \in \N} 
$$
and
$$
	 \mathcal{A}^{a,b}_{\{\omega\}, \{\eta\}} = ((e^{\frac{1}{N} \omega(b |l|) - \frac{1}{n} \eta(a |k|)})_{k, l \in \Z})_{n, N \in \N}. 
$$

	\begin{proposition}[{\cite[Proposition 7.6]{D-N-WeighPLBUltradiffFuncMultSp}}]
		\label{p:GaborCont}
		Let $\psi \in \mathcal{S}^{[\omega]}_{[\eta]}$ and let $a,b> 0$. Then, the mappings
		$$
				C_{\psi}^{a,b} : \mathcal{Z}^{[\omega]}_{[\eta]} \to \lambda^\infty(\mathcal{A}^{a,b}_{[\omega], [\eta]}) \qquad  \mbox{and} \qquad D_{\psi}^{a,b} : \lambda^\infty(\mathcal{A}^{a,b}_{[\omega], [\eta]}) \to \mathcal{Z}^{[\omega]}_{[\eta]}
				$$
		are continuous.
			\end{proposition}

\subsection{Sequence space representations of $\mathcal{Z}^{(\omega)}_{(\eta)}$}

Given a weight function $\omega$, we define $ \alpha(\omega)= (\omega(i))_{i \in \N}$. Conditions $(\alpha)$ and $(\gamma)$ imply that $\alpha(\omega)$ is a stable exponent sequence that satisfies \eqref{nucls} and thus (N). We are ready to show the second main result of this article.

	\begin{theorem}
		\label{t:SeqSpRep}
		Let $\omega$  be a BMT-weight function and let $\eta$ be a weight function  such that $\mathcal{S}^{(\omega)}_{(\eta)}$ is Gabor accessible.
		Then, 
	\begin{equation}
	\label{SSR}
		\mathcal{Z}^{(\omega)}_{(\eta)} \cong \Lambda_\infty(\alpha(\omega), \alpha(\eta)).
	\end{equation}
	\end{theorem}
	\begin{proof}
	As  $\mathcal{S}_{(\eta)}^{(\omega)}$ is Gabor accessible, there are $\psi, \gamma \in \mathcal{S}_{(\eta)}^{(\omega)}$ and $a,b >0$ such that $(\psi,\gamma)$ is a pair of dual windows on $a \Z \times b \Z$. 
Condition $(\alpha)$ implies that 
$$
\Lambda_\infty( \alpha(\omega), \alpha(\eta)) \cong  \lambda^\infty(\mathcal{A}^{a,b}_{(\omega), (\eta)}) \cong \lambda^\infty(\mathcal{A}^{1/b,1/a}_{(\omega), (\eta)}).
$$  
Hence, by Theorem \ref{t:PelczynskiPLS}, it suffices to verify that $\mathcal{Z}_{(\eta)}^{(\omega)} \compl \lambda^\infty(\mathcal{A}^{a,b}_{(\omega), (\eta)})$ and $\lambda^\infty(\mathcal{A}^{1/b,1/a}_{(\omega), (\eta)}) \compl \mathcal{Z}_{(\eta)}^{(\omega)}$. Proposition \ref{p:GaborCont} shows that the mappings
$$
 C^{a,b}_\psi: \mathcal{Z}_{(\eta)}^{(\omega)} \rightarrow \lambda^\infty(\mathcal{A}^{a,b}_{(\omega), (\eta)}) \qquad \mbox{and} \qquad 
D^{a,b}_\gamma:  \lambda^\infty(\mathcal{A}^{a,b}_{(\omega), (\eta)}) \to \mathcal{Z}_{(\eta)}^{(\omega)}
$$
are continuous, and in view of \eqref{comp11}, it holds that $D^{a,b}_\gamma \circ C^{a,b}_\psi = \operatorname{id}_{ \mathcal{Z}_{(\eta)}^{(\omega)}}$. This shows that  $\mathcal{Z}_{(\eta)}^{(\omega)} \compl \lambda^\infty(\mathcal{A}^{a,b}_{(\omega), (\eta)})$ . Another application of Proposition \ref{p:GaborCont} gives that the mappings
$$
 C^{\frac{1}{b}, \frac{1}{a}}_\psi:  \mathcal{Z}_{(\eta)}^{(\omega)} \rightarrow \lambda^\infty(\mathcal{A}^{1/b,1/a}_{(\omega), (\eta)}) \qquad \mbox{and} \qquad 
D^{\frac{1}{b}, \frac{1}{a}}_\gamma:  \lambda^\infty(\mathcal{A}^{1/b,1/a}_{(\omega), (\eta)}) \to  \mathcal{Z}_{(\eta)}^{(\omega)}
$$
are continuous, and by \eqref{WR}, it holds that $(ab)^{-1}C^{\frac{1}{b}, \frac{1}{a}}_\psi \circ D^{\frac{1}{b},\frac{1}{a}}_\gamma =  \operatorname{id}_{   \lambda^\infty(\mathcal{A}^{1/b,1/a}_{(\omega), (\eta)})}$. 
Therefore $\lambda^\infty(\mathcal{A}^{1/b,1/a}_{(\omega), (\eta)}) \compl \mathcal{Z}_{(\eta)}^{(\omega)}$.
\end{proof}

\begin{corollary} \label{main-text} Let $\omega$  be a BMT-weight function and let $\eta$ be a weight function  such that  $\omega(t) = o(t^2)$ and $\eta(t) = o(t^2)$. Then, the sequence space representation \eqref{SSR} holds.
\end{corollary}
\begin{proof}
This follows from Proposition \ref{GA-1} and Theorem \ref{t:SeqSpRep}.
\end{proof}
For $\omega=\omega_\mu$ and $\eta = \omega_\tau$ with $\mu,\tau > 1/2$, Corollary \ref{main-text} implies  Theorem \ref{main-intro} from the introduction (see Example \ref{CGS2}).

We end this article with several remarks. 
\begin{remark} Let $\omega$  be a BMT-weight function and let $\eta$ be a weight function
\begin{enumerate}
\item We believe that the sequence space representation \eqref{SSR} holds under the assumption that $\mathcal{S}^{(\omega)}_{(\eta)} \neq \{0\}$, but are unable to show this. This would follow from Theorem \ref{t:SeqSpRep} if it holds that  every non-trivial $\mathcal{S}^{(\omega)}_{(\eta)}$ is Gabor accessible (see Remark \ref{GAR}).
\item Similarly to the sequence space representation \eqref{SSR}, we believe that in the Roumieu case
\begin{equation}
\label{SSRR}
\mathcal{Z}^{\{\omega\}}_{\{\eta\}} \cong \Lambda_0(\alpha(\eta), \alpha(\omega)).
\end{equation}
In the case that $\mathcal{S}^{\{\omega\}}_{\{\eta\}}$ is Gabor accessible, by using the same argument as in Theorem \ref{t:SeqSpRep},  this would be implied by a positive solution to the question posed in Remark \ref{openqabst}. 
\item Let $\psi \in L^2(\R)$. We define
  \begin{align*}
    \psi_{0,l} &= T_{l} \psi, \qquad l \in \Z, \\
    \psi_{k,l} &= \frac{1}{\sqrt{2}} T_{\frac{l}{2}}(M_k+ (-1)^{k+l} M_{-k})\psi, \qquad (k,l)\in\N_{>0} \times \Z ,
  \end{align*}
and set $\mathcal{W}(\psi) = \{   \psi_{k,n} \mid  (k,n)\in\N\times\Z \}$. We call  $\mathcal{W}(\psi)$ a \emph{Wilson basis} if it is an orthonormal basis in $L^2(\R)$. See  \cite[Section 8.5]{G-FoundTFAnalysis} for more information.

A celebrated result of Daubechies, Jaffard, and Journ\'e \cite[Theorem 4.1]{D-J-J-WilsonBasisExpDecay} states that there  exists $\psi \in \mathcal{S}^{1}_{1}$ such that  $\mathcal{W}(\psi)$ is a Wilson basis. Furthermore, if $\omega$ or $\eta$ is non-quasianalytic, then $\mathcal{S}^{[\omega]}_{[\eta]}$
contains an element $\psi$ such that $\mathcal{W}(\psi)$ is a Wilson basis (cf.\ \cite[Corollary 8.5.5(b)]{G-FoundTFAnalysis}). If there exists $\psi \in \mathcal{S}^{[\omega]}_{[\eta]}$ such that  $\mathcal{W}(\psi)$ is a Wilson basis, then \cite[Corollary 8.5.4]{G-FoundTFAnalysis} implies that $\mathcal{S}^{[\omega]}_{[\eta]}$ is Gabor accessible.
 
Suppose that there exist $\psi \in \mathcal{S}^{[\omega]}_{[\eta]}$ such that  $\mathcal{W}(\psi)$ is a Wilson basis. By using Proposition \ref{p:GaborCont} and a similar argument as in the proofs of \cite[Proposition 4.8 and Theorem 5.4]{D-N-SeqSpRepBBSp}, one can show that 
$$
\Psi: \mathcal{Z}^{(\omega)}_{(\eta)}  \to \Lambda_\infty(\alpha(\eta), \alpha(\omega))  \qquad \mbox{and} \qquad \Psi:\mathcal{Z}^{\{\omega\}}_{\{\eta\}}  \to \Lambda_0(\alpha(\eta), \alpha(\omega))
$$
are topological isomorphisms, where $\Psi(f) = (c_{k,n})_{(k,n) \in \N^2}$ is defined via
$$
    c_{k,n}=
    \begin{cases}
    \langle f, \psi_{-\frac{k}{2},n} \rangle, & k\;\text{even}, \\
      \langle f, \psi_{\frac{k+1}{2},n} \rangle, & k\;\text{odd}.
    \end{cases}
$$
This gives a constructive proof of the sequence space representations \eqref{SSR} and \eqref{SSRR}. 
\end{enumerate}
\end{remark}

\end{document}